\documentclass[12pt,preprint]{amsart}

\newtheorem{assump}{Assumption}

\usepackage{amsmath, latexsym, amsfonts, amssymb, amsthm}
\usepackage{graphicx, color,hyperref,dsfont,epsfig,caption,wrapfig,subfig}
\usepackage{pstricks,yfonts,mathalfa}
\hypersetup{
    linktoc=page,
    linkcolor=red,          
    citecolor=blue,        
    filecolor=blue,      
    urlcolor=cyan,
    colorlinks=true           
}

\numberwithin{equation}{section}

\newtheorem{theorem}{Theorem}[section]
\newtheorem{lemma}[theorem]{Lemma}
\newtheorem{proposition}[theorem]{Proposition}
\newtheorem{corollary}[theorem]{Corollary}

\newtheorem{definition}[theorem]{Definition}

\theoremstyle{definition}

\renewcommand{\tilde}{\widetilde}          
\DeclareMathSymbol{\leqslant}{\mathalpha}{AMSa}{"36} 
\DeclareMathSymbol{\geqslant}{\mathalpha}{AMSa}{"3E} 
\DeclareMathSymbol{\eset}{\mathalpha}{AMSb}{"3F}     
\renewcommand{\leq}{\;\leqslant\;}                   
\renewcommand{\geq}{\;\geqslant\;}                   
\newcommand{\dd}{\text{\rm d}}             

\newcommand{\R}{\mathbb{R}}

\newcommand{\N}{\mathbb{N}}

\newcommand{\E}{\mathds{E}}
\renewcommand{\P}{\mathds{P}}

\def\bi{\begin{itemize}}
\def\ei{\end{itemize}}

\def\bnum{\begin{enumerate}}
\def\enum{\end{enumerate}}
\def\<#1{\langle #1 \rangle}

\usepackage{bm}
\usepackage{bbm}

\title{Hölder type estimates for Gaussian multiplicative chaos}

\author{Yulai Huang}
\email{yulai.huang@univ-amu.fr}
\address{Aix-Marseille Université, I2M, CNRS, Marseille, France}

\begin{document}

\begin{abstract}
    We investigate the right tail behavior
    of a certain class of GMC ratios, reminiscent of Hölder's inequality.
    We start with a heuristic argument to
    justify the optimal
    exponent in the tail estimate.
    Since Kahane's convexity inequality
    \cite{Kahane85} does not apply to GMC ratios,
    implementing the heuristic
    in the continuous setting is nontrivial
    from the viewpoint of GMC theory.
    We address the problem by enlarging the
    class of GMC ratios considered,
    and deduce the upper and lower bounds
    for the right tail of GMC ratios.
\end{abstract}

\maketitle

\section{Introduction}


Let $h$ be a log-correlated Gaussian field
on a bounded domain $\Omega\subset\R^d$ with covariance kernel
$$
  \E[h(x)h(y)]=-\ln|x-y|+g(x,y),
$$
where $g$ is continuous on $\Omega\times\Omega$.
For a parameter $\gamma>0$, the Gaussian multiplicative chaos (GMC)
associated to $h$ is a random measure
on $\Omega$, formally denoted as
$$
  M_{\gamma,h}(\dd x)=e^{\gamma h(x)-\frac{\gamma^2}{2}\E[h(x)^2]}\;\dd x.
$$
Since $h$ is only a random distribution, the measure $M_{\gamma,h}$
should be defined as a limit of $M_{\gamma,h_{\epsilon}}$
where $(h_{\epsilon})_{\epsilon>0}$ is a suitable regularization of $h$.

The study of GMC theory was initiated by Kahane in the 1980s.
In his seminal work \cite{Kahane85},
Kahane established the uniqueness of
the limiting measure for a special class of log-correlated kernels and regularization schemes,
and proved that $M_{\gamma,h}$ is non-degenerate
if and only if $\gamma^2<2d$.
When $h_{\epsilon}$ is the convolution of $h$ with some mollifier,
Robert and Vargas showed the convergence in law of $M_{\gamma,h_{\epsilon}}$,
extending Kahane's construction to general log-correlated kernels \cite{Robert2008GaussianMC}.
See also \cite{Berestycki17} for a simple and elegant treatment
of the convergence issues by Beretycki,
and \cite{SHAMOV20163224} for another viewpoint of GMC developed by Shamov.
A classical result in GMC theory states that $M_{\gamma,h}(\Omega)$
admits negative moments of all orders.
This was proved in \cite{Robert2008GaussianMC} and
was extended to GMC on fractal sets in \cite{10.1214/18-ECP168}.

Recently, the study of Liouville conformal field theory and related models
in mathematical physics has brought light to several interesting problems
in GMC theory.
For instance,
as an intermediate step in the proof of the \textit{KPZ relation},
Duplantier and Sheffield established
an upper bound for the small tail of GMC (Lemma 4.5 in \cite{DuplantierSheffield11}):
If $\Omega$ is the unit disk and $h$ is a zero-boundary Gaussian free field on $\Omega$, then there exist constants $C,c>0$ such that
$$
  \P(M_{\gamma,h}(\Omega)<x)\leq Ce^{-c(\ln x)^2},\quad \forall x\in(0,1).
$$
In particular, it implies the existence of negative moments of all orders for $M_{\gamma,h}(\Omega)$.
This estimate is optimal for generic GMC measures.
Indeed, if $h=h'+N$ where $N$ is a standard Gaussian variable independent of $h'$, then $M_{\gamma,h}(\Omega)$ can be decomposed as
$$
  M_{\gamma,h}(\Omega)=
  \exp^{\gamma N-\frac{\gamma^2}{2}}M_{\gamma,h'}(\Omega)
$$
and the ``macroscopic'' term $e^{\gamma N}$ dominates
the small tail behavior of $M_{\gamma,h}(\Omega)$.

Using techniques in Liouville conformal field theory,
Remy obtained an exact formula for the density of
the total mass $Y_{\gamma}$ of a certain GMC on the unit circle \cite{Remy20},
answering a conjecture of Fyodorov and Bouchaud \cite{Fyodorov_2008}.
The formula reveals a much sharper small tail behavior for this specific GMC:
$$
  \P(Y_{\gamma}<x)\leq C\exp(-cx^{-4/\gamma^2}),\quad\forall x\in(0,1).
$$
This further implies the existence of exponential moments for $Y_{\gamma}^{-1}$.
It turns out that the same phenomenon holds for a general class of GMC measures.
Motivated by the probabilistic construction of quantum Mabuchi theory,
Lacoin, Rhodes and Vargas studied the
small tail behavior of $M_{\gamma,h}(\Omega)$ under the assumption that the spatial average of $h$ vanishes.
More precisely,
let $m_h(\Omega)=|\Omega|^{-1} \int_\Omega h(x)\;\dd x$,
then the following holds for $\zeta$ sufficiently large \cite{mabuchi}:
$$
  \P(e^{-\gamma m_h(\Omega)} M_{\gamma,h}(\Omega)<s)\leq 2 \exp(-c|\log s|^{-\zeta}s^{-4/\gamma^2}),\quad
  \forall s>0.
$$
In particular, $e^{\gamma m_h(\Omega)}M_{\gamma,h}(\Omega)^{-1}$
has finite exponential moments,
which is essential in the construction
of the Mabuchi path integral in \cite{mabuchi}.
Further work by Barashkov, Oikarinen and Wong has removed the $|\log s|$ term
in the estimate and established a corresponding lower bound on the torus
\cite{barashkov2025smalldeviationsgaussianmultiplicative}.
This suggests that the exponent $4/\gamma^2$ is optimal.

In the study of Liouville conformal field theory with boundary \cite{guillarmou2024conformalbootstrapsurfacesboundary},
it is speculated that the theory can still be defined
when boundary cosmological constants are negative.
This amounts to establishing a sub-exponential
right tail estimate for the ratio of two GMC
with parameter $\gamma/2$ and $\gamma$.
We consider in this paper an analogue of this problem.
For $\alpha,\gamma<\sqrt{2d}$ distinct
and for any Borel set $S\subset\Omega$,
define the balanced ratio of GMC associated to $h$ by
$$
  Q_h(S)=\frac{M_{\alpha,h}(S)^{\frac{\gamma}{\gamma-\alpha}}}
  {M_{\gamma,h}(S)^{\frac{\alpha}{\gamma-\alpha}}}.
$$
The goal of this paper is to study the behavior of
$\P(Q_h(S)>x)$ as $x\to\infty$.

Note that if $h=h'+N$ for some Gaussian variable $N$ independent of $h$,
the quotient $Q_{h}(S)/Q_{h'}(S)$
does not contain any macroscopic term of the form $e^{\beta N}$.
This explains the terminology \textit{balanced}.
Moreover, for any function $f$ on $S$,
Hölder's inequality gives
\begin{equation}\label{real_holder}
  \frac{(\int_S e^{\alpha f(x)}\;\dd x)^{\frac{\gamma}{\gamma-\alpha}}}
  {(\int_S e^{\gamma f(x)}\;\dd x)^{\frac{\alpha}{\gamma-\alpha}}}\leq |S|
\end{equation}
where $|S|$ is the Lebesgue measure of $S$.
We can view $Q_h(S)$ as a probabilistic counterpart
of the above deterministic ratio,
with function $f$ replaced by Gaussian field $h$.
Since $\E[h(x)^2]$ diverges,
Hölder's inequality does not give us an upper bound for $Q_h(S)$.
Nevertheless, we can establish a sub-exponential
estimate for $\P(Q_h(S)>x)$ (see Theorem \ref{upper}),
reminiscent of Hölder's inequality in the deterministic setting.

\textit{Conventions and notations.}
We denote by $a\wedge b$ and $a\vee b$ the minimum and maximum of $a$ and $b$,
respectively.
We also write $x_+=x\vee 0$ and
$\ln_+x=(\ln x)\vee 0$.
We call a function smooth if it is infinitely differentiable ($C^{\infty}$).
Throughout this paper, all Gaussian fields are assumed to be centered.

\subsection{Setup and main results}

Let $S=[0,1]^2$ be a unit square.
Consider a Gaussian field $h$ defined on a neighborhood $\Omega\subset\R^2$ of $S$
with covariance kernel
$$
  K(x,y)=-\ln|x-y|+g(x,y),
$$
where $g$ is smooth on $\Omega\times \Omega$.
The main result of this paper is the following
upper bound for the right tail of $Q_h(S)$.

\begin{theorem}\label{upper}
For any distinct $\alpha,\gamma\in(0,2)$ and
$0<\delta<\frac{4}{\alpha\gamma}$,
there exists some constants $C,c>0$ such that
$$
  \P(Q_h(S)>x)
  \leq C\exp(-c x^{\frac{4}{\alpha\gamma}-\delta}),\quad\forall x>0.
$$
\end{theorem}

At a formal level, the exponent $4/\alpha\gamma$
given by the above estimate aligns with the
small deviation result for GMC
and the conjectured left tail estimate for
derivative GMC in \cite{mabuchi}.
In fact, we can rewrite Theorem \ref{upper} as
$$
  \ln \P(M_{\alpha,h}(S)^{-\frac{\gamma}{\alpha}}
  M_{\gamma,h}(S)<x)\lesssim -x^{-\frac{4}{\gamma(\gamma-\alpha)}+\delta}
  \quad\text{as}\quad x\to 0.
$$
As $\alpha\to 0$ we have $M_{\alpha,h}(S)^{\frac{1}{\alpha}}\to e^{m_h(S)}$.
Then, the above estimate becomes
$$
  \ln\P(e^{-\gamma m_h(S)}M_{\gamma,h}(S)<x)\lesssim -x^{-\frac{4}{\gamma^2}+\delta}
  \quad\text{as}\quad x\to 0,
$$
whose exponent matches the small deviation result
mentioned earlier.
Moreover, from the elementary inequality
$\theta x+(1-\theta)y\geq x^{\theta}y^{1-\theta}$,
we deduce
$$
  -\alpha^{-1} Q_h(S)\leq
  (\gamma-\alpha)^{-1}
  (M_{\gamma,h}(S)-M_{\alpha,h}(S)).
$$
As $\alpha\to\gamma$, the right hand side tends to
the so-called derivative martingale $M_{\gamma,h}(S)'$. Then
$$
  \ln\P(M_{\gamma,h}(S)'<-x)\leq
  \lim_{\alpha\to\gamma^-}
  \ln\P(Q_h(S)>\alpha x)
  \lesssim -x^{\frac{4}{\gamma^2}-\delta}
  \quad \text{as}\quad x\to\infty.
$$
This matches the conjectured left tail behavior $\ln\P(M_{\gamma,h}'(S)<-x)\lesssim -x^{\frac{4}{\gamma^2}}$ in \cite{mabuchi}.
However, the above discussions remain far from
being rigorous. In the setting of
branching random walks,
the conjecture is proved in \cite{bonnefont2023lefttailsubcriticalderivative,chen2025lefttailsubcriticalderivative},
while the general case is still open.

We also establish a lower bound for the right tail of $Q_h(S)$,
which suggests that the exponent given in Theorem \ref{upper} is optimal.

\begin{theorem}\label{lower}
  Suppose that $h$
  admits a decomposition $h=X+Z$ where $X$ is a star-scale invariant Gaussian field
  (see Section \ref{sec_2} for a precise definition)
  and $Z$ is independent of $X$.
  Then for any distinct $\alpha,\gamma\in(0,2)$ and
  $\delta>0$,
  there exists $C,c>0$ such that
  $$
    \P(Q_h(S)>x)
    \geq c\exp(-C x^{\frac{4}{\alpha\gamma}+\delta}),\quad\forall x>0.
  $$
\end{theorem}

For ease of notation, the statement of Theorem \ref{lower} is restricted
to covariance kernels that can be written as
the sum of a star-scale invariant kernel
and a smooth kernel.
However, the argument is valid for more general log-correlated kernels with a nice white noise decomposition.

\subsection{Organization of the paper}

In Section \ref{sec_2},
we introduce the notion of star-scale invariant
kernels and provide a heuristic proof
of Theorem \ref{upper} for these covariance kernels.
In Section \ref{sec_3},
we show that balanced ratios admit
positive moments of all orders.
In Section \ref{sec_4}, a rigorous version of the heuristic is implemented,
which results in bounds on the Laplace transform
of balanced ratios.
In Section \ref{sec_5}, we establish the
main theorems using results from the previous sections.
In the Appendix,
statements and proofs that deviate from the main line of argument are included.
Notably, we establish two Kahane-type inequalities
for balanced ratios
which is used repeatedly in our analysis.

\subsection{Acknowledgements}
The author wishes to thank Rémi Rhodes
for suggesting this problem
and for the instructive and fruitful discussions.
The author is partially funded by the French ANR grant ANR-21-CE40-0003 and the Simons Fondation Award SFI-MPS-PP-00012621-04.

\section{Backgrounds and heuristics}\label{sec_2}


In this section, we introduce the notion of star-scale invariant kernels,
a class of log-correlated kernels which will play
a particularly important role in this paper.
Then we outline a heuristic proof of Theorem \ref{upper}
for these covariance kernels.

\subsection{Star-scale invariant kernels}

Let $\rho\colon \R^d\to \R$ be a positive definite function,
which is radial, smooth, supported in the unit ball and $\rho(0)=1$.
For example, it can be realized as the convolution
of a compactly supported smooth radial function with itself.

\begin{definition}
  The covariance kernel
  $
    K^0(x,y)=\int_0^{\infty}\rho(e^u(x-y))\;\dd u
  $
  is called a star-scale invariant kernel.
  The associated Gaussian field $X$
  is called a star-scale invariant field.
\end{definition}

The star-scale invariant field $X$ admits a
white noise decomposition $(X_t(x))_{t\geq 0,x\in\R^2}$,
that is a smooth Gaussian field on $[0,\infty)\times\R^d$
with covariance kernel
$$
  \E[X_t(x)X_s(y)]=\int_0^{t\wedge s}\rho(e^u(x-y))\;\dd u.
$$
It is straightforward to see that $X_t-X_s$ is independent of $X_s$ for $t>s$, and that the law of $X_t$
is invariant under translation.
The lemma below explains the name \textit{star-scale invariant}.

\begin{lemma}\label{star_scale}
  The function $g_0(x,y)=K^0(x,y)+\ln|x-y|$
  extends to a smooth function on $\R^d\times\R^d$.
  In addition, we have the following properties:
\begin{itemize}
  \item[(i)] There exists $A>0$ such that
  $$
    |\E[X_t(x)X_t(y)]-t\wedge \ln_+|x-y|^{-1}|<A,\quad\forall
    t\geq 0,\forall x,y\in \R^d.
  $$
  \item[(ii)] For any $t>s$, we have an equality in law
  $$
    (X_t(x)-X_s(x))
    _{x\in\R^d}\stackrel{\text{law}}
    {=}(X_{t-s}(e^s x))_{x\in\R^d}.
  $$
  \item[(iii)] For any $t>s$ and $|x-y|\geq e^{-s}$,
  $X_t(x)-X_s(x)$ and $X_t(y)-X_s(y)$ are independent.
\end{itemize}
\end{lemma}

\begin{proof}
    We only show the smoothness of $g_0$.
    The rest of the properties are left to the readers.
    Since $\rho$ is smooth and radial,
    there exists a smooth function $f$ on $\R$ such that $\rho(x)=f(|x|^2)$.
    Then $f(0)=1$ and $\operatorname{supp} f\subset(-\infty,1]$.
    By a change of variable $t=e^{2u}|x-y|^2$ we obtain
    $$
      g_0(x,y)
      =\int_{|x-y|^2}^1
      \frac{f(t)-1}{2t}\;\dd t.
    $$
    Since the integrand is smooth in $t$ and
    the mapping $(x,y)\mapsto |x-y|^2$
    is smooth on $\R^d\times\R^d$,
    $g_0$ extends smoothly on the diagonal.
\end{proof}

For the rest of the paper,
fix a star-scale invariant field $X$
and a white noise decomposition $(X_t(x))_{t\geq 0,x\in\R^d}$.
We adopt the notation
$$
    G_{\gamma,t}=M_{\gamma,X_t},
    \quad G_{\gamma,t}^{(s)}=M_{\gamma,X_t-X_s}.
$$
Denote also $G_{\gamma}=M_{\gamma,X}$ and $G^{(s)}_{\gamma}=M_{\gamma,X-X_s}$.

\subsection{Cascade relation}

Let $S=[0,1]^d$ be a $d$-dimensional cube
and $\alpha,\gamma\in(0,\sqrt{2d})$.
We give here a heuristic justification
of the expected tail behavior:
$$
  \ln\P(Q_X(S)>x)\lesssim -x^{\frac{2d}{\alpha\gamma}}\quad\text{as}\quad x\to\infty.
$$

First, observe that $Q_X$ is subadditive
on Borel sets, i.e. $Q_X(B\cup B')\leq Q_X(B)+Q_X(B')$ for $B$ and $B'$ disjoint. 
In fact, it is a consequence of Hölder's inequality
\begin{equation}\label{holder}
  G_{\alpha}(B)+G_{\alpha}(B')\leq
  (G_{\gamma}(B)+G_{\gamma}(B'))
  ^{\frac{\alpha}{\gamma}}
  (Q_X(B)+Q_X(B'))
  ^{\frac{\gamma-\alpha}{\gamma}}.
\end{equation}
Fix an even integer $e^s$.
Divide $S$ evenly into $e^{ds}$ small cubes
$\{S_1,\ldots,S_{e^{ds}}\}$ of side length $e^{-s}$.
Notice that
$$
  M_{\gamma,X}(S_i)=\lim_{t\to\infty}
  \int_{S_i} e^{\gamma X_t(x)-\frac{\gamma^2}{2}t}\;\dd x
  =\lim_{t\to\infty}
  \int_{S_i} e^{\gamma X_s(x)-\frac{\gamma^2}{2}s}\;
  G_{\gamma,t}^{(s)}(\dd x)=
  \int_{S_i} e^{\gamma X_s(x)-\frac{\gamma^2}{2}s}\;G^{(s)}_{\gamma}(\dd x).
$$
Let $\phi_i(x)=e^{-s}x+a_i$ be the affine map
such that $\phi_i(S)=S_i$.
Define a Gaussian field $X^{(i)}$ on $S$ by $X^{(i)}=(X-X_s)\circ \phi_i$.
Using the subadditivity of $Q_X$
and the change of variable $x_i=\phi_i(x)$, we obtain
\begin{equation}\label{subad}
  \begin{aligned}
  Q_X(S)\leq \sum_{i=1}^{e^{ds}}
  Q_X(S_i)&=
  e^{\frac{\alpha\gamma}{2}s}\sum_{i=1}^{e^{ds}}
  \frac{(\int_{S_i} e^{\alpha X_s(x_i)}\;G^{(s)}_{\alpha}(\dd x_i))^{\frac{\gamma}{\gamma-\alpha}}}
  {(\int_{S_i} e^{\gamma X_s(x_i)}\;G^{(s)}_{\gamma}(\dd x_i))^{\frac{\alpha}{\gamma-\alpha}}}\\
  &=e^{(\frac{\alpha\gamma}{2}-d)s}\sum_{i=1}^{e^{ds}}
  \frac{(\int_S e^{\alpha X_s(\phi_i(x))}\;M_{\alpha,X^{(i)}}(\dd x))^{\frac{\gamma}{\gamma-\alpha}}}
  {(\int_S e^{\gamma X_s(\phi_i(x))}\;M_{\gamma,X^{(i)}}(\dd x))^{\frac{\alpha}{\gamma-\alpha}}}.
  \end{aligned}
\end{equation}

To simplify the discussion we make three assumptions.

\begin{assump}\label{assump1}
  $X_s|_{S_i}=X_s(x_i)$
  with $x_i$ the barycenter of $S_i$.
\end{assump}

\begin{assump}\label{assump2}
  $X^{(i)}$ has the same law as $X|_S$.
\end{assump}

\begin{assump}\label{assump3}
  $X^{(i)}$ are mutually
  independent for $i=1,\ldots,e^{ds}$.
\end{assump}
\noindent
These assumptions correspond to Lemma \ref{star_scale} (i)(ii)(iii) respectively.
In fact, the covariance kernel of $X_s$ on each cube $S_i$
satisfies
$
  |\E[X_s(x)X_s(y)]-s|\leq A
$
for some universal constant $A$ in (i).
Hence, we can assume that $X_s$ is constant on each cube $S_i$ and Assumption \ref{assump1} follows.
Assumption \ref{assump2} is a directly consequence of (ii).
By (iii), 
$X^{(i)}$ and $X^{(j)}$ are independent
if $S_i$ and $S_j$ are disjoint,
which leads to Assumption \ref{assump3}.

Under Assumption \ref{assump1},
$e^{\gamma X_s(\phi_i(x))}$ and $e^{\alpha X_s(\phi_i(x))}$
on the right hand side of (\ref{subad})
cancel 
thanks to the choice of exponents.
Then, (\ref{subad}) becomes
a stochastic dominance
$$
  Q_X(S)\leq e^{(\frac{\alpha\gamma}{2}-d)s}
  \sum_{i=1}^{e^{ds}}Q_{X^{(i)}}(S)
$$
where $Q_{X^{(i)}}(S)$ are mutually independent
and have the same law as $Q_X(S)$,
thanks to Assumption \ref{assump2} and \ref{assump3}.
We call the inequality above the \textit{cascade relation}
for balanced ratios,
as a similar equation holds for Mandelbrot's multiplicative cascades \cite{KAHANE1976131}.

Assume that $Q_X(S)$ admits positive moments of all orders,
we can deduce from the cascade relation and induction that $Q_X(S)$ has exponential moments.
Then Cramér's large deviation theorem implies
for $\theta>\E[Q_X(s)]$
$$
  \ln \P(Q_X(S)>e^{\frac{\alpha\gamma}{2}s}\theta)\leq
  \ln \P\big(e^{-ds}\sum_{i=1}^{e^{ds}}Q_{X_i}(S)>\theta\big)
  \lesssim -e^{-ds}.
$$
This gives us an expected tail estimate $\ln\P(Q_X(S)>x)\lesssim -x^{\frac{2d}{\alpha\gamma}}$
for star-scale invariant kernels.
For $d=2$ the exponent given here matches Theorem \ref{upper}.

It is worth pointing out that, in our problem,
Assumption \ref{assump1} is not very straightforward to justify.
A standard technique in GMC theory is to apply
Kahane's inequality \cite{Kahane85}
to compare the GMC associated to $X_s$ and that
associated to $X_s(x_i)$.
For balanced ratios, this amounts to replacing $X_s|_{S_i}$
by $$X_s(x_i)\pm\sup_{x\in S_i}|X_s(x)-X_s(x_i)|.$$
However, such a replacement introduces a macroscopic factor $\exp(\sup_{x\in S_i}|X_s(x)-X_s(x_i)|)$
in the cascade relation,
whose tail is too heavy for our purposes.

\section{Positive moments of balanced ratios}\label{sec_3}


As a preliminary result, we show in this section
the existence of positive moments of all orders
for balanced ratios.
For this section,
$S=[0,1]^d$ is a unit cube
and $0<\alpha<\gamma<\sqrt{2d}$.
The main result is the following theorem.

\begin{theorem}\label{moment_gen}
  Let $h$ be a Gaussian field on a neighborhood of $S$
  with covariance kernel $$\E[h(x)h(y)]=-\ln|x-y|+g(x,y),$$
  where $g$ is continuous.
  Then for any $0<\alpha<\gamma<\sqrt{2d}$,
  $$
    \E[Q_h(S)^n]<\infty,
    \quad \forall n\geq 0.
  $$
\end{theorem}

We provide below an alternative formulation of the theorem,
which will be useful in the next section.
Recall that $G_{\gamma,t}$ is the GMC associated
to the white noise decomposition $(X_t(x))_{t\geq 0,x\in\R^d}$ defined in Section \ref{sec_2}.

\begin{proposition}\label{moment_fini}
  For any $0<\alpha<\gamma<\sqrt{2d}$, we have
  $$
    \sup_{t\geq 0}\E\Big[\frac{G_{\alpha,t}(S)^n}
    {G_{\gamma,t}(S)
    ^{\frac{n\alpha}{\gamma}}}\Big]<\infty,\quad\forall n\geq 0.
  $$
\end{proposition}

\subsection{Scaling relations for GMC ratios}

To begin with,
we prove the scaling relation for GMC ratios associated to star-scale invariant kernels.
For $p,q\in\R$, define the scaling exponent
$$
  \zeta(p,q)=\frac{1}{2}\alpha^2p^2-(\frac{\alpha^2}{2}+d+\alpha\gamma q)p
  +(\frac{\gamma^2}{2}+d)q+\frac{\gamma^2}{2}q^2.
$$

\begin{lemma}\label{scale}
  Fix $R>0$.
  For any Borel set $B\subset \{x\in\R^d:|x|\leq R\}$
  and $s<t$, we have
  $$
    \E\Big[\frac{G_{\alpha,t}(e^{-s} B)^p}
    {G_{\gamma,t}(e^{-s} B)^q}\Big]
    \asymp e^{\zeta(p,q)s}
    \E\Big[\frac{G_{\alpha,t-s}(B)^p}
    {G_{\gamma,t-s}(B)^q}\Big]
  $$
  where $P\asymp P'$ is used to denote $C^{-1}\leq P/P'\leq C$ for some constant
  $C$ determined by $R,p,q$.
\end{lemma}

\begin{proof}
  We can write
  $$
    \frac{G_{\alpha,t}(e^{-s} B)^p}
    {G_{\gamma,t}(e^{-s} B)^q}
    =\frac{\big(\int_{e^{-s} B}e^{\alpha X_s-\frac{\alpha^2}{2}s}\;G_{\alpha,t}^{(s)}(\dd x)\big)^p}
    {\big(\int_{e^{-s} B}e^{\gamma X_s-\frac{\gamma^2}{2}s}\;G_{\gamma,t}^{(s)}(\dd x)\big)^q}\;.
  $$
  By Lemma \ref{star_scale} (i),
  there exists $A>0$ determined by $R$, such that
  $$|\E[X_s(x)X_s(y)]-s|\leq A
  ,\quad \forall s\geq 0\ \forall x,y\in e^{-s}B.$$
  Applying Proposition \ref{kahane_variant}
  to $X_s$ and a Gaussian variable $N$ with variance $s$, we obtain
  \begin{align*}
    \E\Big[\frac{G_{\alpha,t}(e^{-s} B)^p}
    {G_{\gamma,t}(e^{-s} B)^q}\Big]
    \asymp \E\Big[\frac{(\int_{e^{-s} B}e^{\alpha N
    -\frac{\alpha^2}{2}s}\;G_{\alpha,t}^{(s)}(\dd x))^p}
    {(\int_{e^{-s} B}e^{\gamma N-\frac{\gamma^2}{2}s}\;
    G_{\gamma,t}^{(s)}(\dd x))^q}\Big]
    =e^{\frac{1}{2}(\alpha p-\gamma q)^2s
    +\frac{1}{2}(\gamma^2q-\alpha^2p)s}
    \E\Big[\frac{G_{\alpha,t}^{(s)}(e^{-s} B)^p}
    {G_{\gamma,t}^{(s)}(e^{-s} B)^q}\Big].
  \end{align*}
  Since $X_t-X_s$ has the same law of $X_{t-s}(e^s\cdot)$,
  we have by a change of variable
  $$
    \E\Big[\frac{G_{\alpha,t}^{(s)}(e^{-s} B)^p}
    {G_{\gamma,t}^{(s)}(e^{-s} B)^q}\Big]
    =e^{(q-p)ds}
    \E\Big[\frac{G_{\alpha,t-s}(B)^p}
    {G_{\gamma,t-s}(B)^q}\Big].
  $$
  Combining above two lines we arrive at the scaling exponent $\zeta(p,q)$.
\end{proof}

\subsection{Proof of Proposition \ref{moment_fini}
}

Set $\epsilon=\frac{d}{\alpha^2}+\frac{1}{2}$.
For $n\in\N$ fixed, $\zeta(p,\frac{n\alpha}{\gamma})$
attains its minimum at $p=n+\epsilon$
and the minimum value is given by
$$
  \zeta_n=(\gamma-\alpha)(\frac{\alpha}{2}-\frac{d}{\gamma})n-\frac{1}{2}(\frac{\alpha}{2}+\frac{d}{\alpha})^2.
$$
Motivated by the above calculation, we show by induction that for all $n\in\N$
\begin{equation}\label{shifted_moment}
  \sup_{t\geq 0}\E\Big[\frac{G_{\alpha,t}(S)^{n+\epsilon}}
  {G_{\gamma,t}(S)
  ^{\frac{n\alpha}{\gamma}}}\Big]<\infty.
\end{equation}
Then Proposition \ref{moment_fini} follows from Hölder's inequality and the fact that
$\sup_{t\geq 0}\E[G_{\alpha,t}^{-\epsilon}]$ is finite (Proposition 3.6 in \cite{Robert2008GaussianMC}).

Since $\epsilon<\frac{2d}{\alpha^2}$,
the initial step $n=0$ is implied by
classical GMC theory.
Suppose that (\ref{shifted_moment}) holds for $n-1$ where $n\geq 1$.
Let $e^s$ be an even integer.
Divide $S$ into $e^{ds}$ small cubes of equal side length $e^{-s}$.
We can partition the small cubes into $2^d$ groups
$\mathcal{S}_1,\ldots,\mathcal{S}_{2^d}$,
such that each group has $N=2^{-d}e^{ds}$ cubes,
and any two cubes in the same group are separated
by a distance of at least $e^{-s}$.
In fact, we can associate a dyadic number of $d$ digits
to each small cube, such
that the $i$-th digits of two cubes differ whenever they share a common face
orthogonal to the $i$-th axis.
Then we split these cubes into $2^d$ groups
based on the designated dyadic number.

We have for $a\geq 2$ and $x,y\geq 0$
\begin{align*}
  (x+y)^a-x^a-y^a&=ay\int_0^1(x+ty)^{a-1}
  -(ty)^{a-1}\;\dd t\\
  &=a(a-1)xy\int_0^1\int_0^1 (rx+ty)^{a-2}\;\dd r\dd t\\
  &\leq a(a-1)xy(x+y)^{a-2}\\
  &\leq a(a-1)2^{a-2}(xy^{a-1}+yx^{a-1}).
\end{align*}
After $N$ iterations, there exists $C_{N,a}>0$
depending only on $N,a$ such that
$$
  (\sum_{i=1}^Nx_i)^a\leq\sum_{i=1}^N x_i^a
  +C_{N,a}\sum_{i\neq j}x_ix_j^{a-1},\quad
  \forall x_1,\ldots,x_N\geq 0.
$$
For a group $\mathcal{S}_l=\{S_1,\ldots,S_N\}$
of small cubes,
we also use $\mathcal{S}_l$ to denote
the disjoint union of its elements.
Since $n+\epsilon>2$,
we have
$$
  G_{\alpha,t}(\mathcal{S}_l)^{n+\epsilon}
  \leq \sum_{i=1}^N G_{\alpha,t}(S_i)^{n+\epsilon}+
  C_{N,n+\epsilon}\sum_{i\neq j}G_{\alpha,t}(S_i)G_{\alpha,t}(S_j)^{n-1+\epsilon}.
$$
Consequently,
$$
  \E\Big[\frac{G_{\alpha,t}(\mathcal{S}_l)^{n+\epsilon}}
  {G_{\gamma,t}(S)^{\frac{n\alpha}{\gamma}}}\Big]
  \leq\sum_{i=1}^N \E\Big[\frac{G_{\alpha,t}(S_i)^{n+\epsilon}}
  {G_{\gamma,t}(S_i)^{\frac{n\alpha}{\gamma}}}\Big]
  +C_{N,n+\epsilon}\sum_{i\neq j}\E\Big[G_{\alpha,t}(S_i)
  G_{\alpha,t}(S_j)^{n-1+\epsilon}
  G_{\gamma,t}(S)^{-\frac{n\alpha}{\gamma}}\Big].
$$
For diagonal terms, since the law of $X_t$ is translation invariant,
we have by Lemma \ref{scale}
\begin{align*}
  \E\Big[\frac{G_{\alpha,t}(S_i)^{n+\epsilon}}
  {G_{\gamma,t}(S_i)^{\frac{n\alpha}{\gamma}}}\Big]
  =\E\Big[\frac{G_{\alpha,t}(e^{-s}S)^{n+\epsilon}}
  {G_{\gamma,t}(e^{-s}S)^{\frac{n\alpha}{\gamma}}}\Big]\leq
  A_n e^{\zeta_n s}\E\Big[\frac{G_{\alpha,t-s}(S)^{n+\epsilon}}
  {G_{\gamma,t-s}(S)^{\frac{n\alpha}{\gamma}}}\Big]
\end{align*}
for some constant $A_n$ determined by $n$.
For off-diagonal terms, since
$|\E[X_s(x)X_s(y)]|\leq s+O(1)$
for $x,y\in S$,
we can use
Kahane-type inequality \ref{kahane_variant}
to compare $X_s$ with the trivial field $0$,
where the underlying Radon measures
are $G_{\alpha,t}^{(s)}$ and $G_{\gamma,t}^{(s)}$.
This results in the inequality
\begin{align*}
  \E[G_{\alpha,t}(S_i)
  G_{\alpha,t}(S_j)^{n-1+\epsilon}
  G_{\gamma,t}(S)^{-\frac{n\alpha}{\gamma}}]
  \leq B_{n,s}\E[G_{\alpha,t}^{(s)}(S_i)G_{\alpha,t}^{(s)}(S_j)^{n-1+\epsilon}
  G_{\gamma,t}^{(s)}(S)^{-\frac{n\alpha}{\gamma}}]
\end{align*}
for some constant $B_{n,s}$ determined
by $n$ and $s$.
Since $S_i$ and $S_j$ are separated by a distance of at least $e^{-s}$,
the restrictions of $X_t-X_s$ on $S_i$ and on $S_j$ are independent.
Note also that $G_{\gamma,t}^{(s)}(S)\geq
G_{\gamma,t}^{(s)}(S_i)\vee
G_{\gamma,t}^{(s)}(S_j)$.
Hence,
\begin{align*}
  &\E[G_{\alpha,t}(S_i)
  G_{\alpha,t}(S_j)^{n-1+\epsilon}
  G_{\gamma,t}(S)^{-\frac{n\alpha}{\gamma}}]\\
  &\leq B_{n,s}\E[G_{\alpha,t}^{(s)}(S_i)G_{\gamma,t}^{(s)}(S_i)^{-\frac{\alpha}{\gamma}}]\;
  \E[G_{\alpha,t}^{(s)}(S_j)^{n-1+\epsilon}
  G_{\gamma,t}^{(s)}(S_j)^{-\frac{(n-1)\alpha}{\gamma}}]\\
  &=B_{n,s} e^{-d(n+\epsilon-\frac{n\alpha}{\gamma})s}
  \E[G_{\alpha,t-s}(S)G_{\gamma,t-s}(S)^{-\frac{\alpha}{\gamma}}]\;
  \E[G_{\alpha,t-s}(S)^{n-1+\epsilon}
  G_{\gamma,t-s}(S)^{-\frac{(n-1)\alpha}{\gamma}}].
\end{align*}
On the last line,
standard GMC theory and Hölder's inequality
imply that the left term is bounded in $t-s$,
while the right term is bounded
in $t-s$ by induction hypothesis.
Consequently the off-diagonal terms
are bounded in $t$.

To summary, for any $t$ and $s$ such that $t\geq s$, we have
$$
  \E\Big[\frac{G_{\alpha,t}(\mathcal{S}_l)^{n+\epsilon}}
  {G_{\gamma,t}(S)^{\frac{n\alpha}{\gamma}}}\Big]
  \leq 2^{-d} A_n e^{(\zeta_n+d) s}\E\Big[\frac{G_{\alpha,t-s}(S)^{n+\epsilon}}
  {G_{\gamma,t-s}(S)^{\frac{n\alpha}{\gamma}}}\Big]+A_{n,s}.
$$
for some constant $A_{n,s}$ determined by $n$ and $s$.
Therefore
\begin{align*}
  \E\Big[\frac{G_{\alpha,t}(S)^{n+\epsilon}}
  {G_{\gamma,t}(S)^{\frac{n\alpha}{\gamma}}}\Big]
  &\leq 2^{d(n+\epsilon)}\sum_{l=1}^{2^d}
  \E\Big[\frac{G_{\alpha,t}(\mathcal{S}_l)^{n+\epsilon}}
  {G_{\gamma,t}(S)^{\frac{n\alpha}{\gamma}}}\Big]\\
  &\leq 2^{d(n+\epsilon-1)} A_n e^{(\zeta_n+d)s}
  \E\Big[\frac{G_{\alpha,t-s}(S)^{n+\epsilon}}
  {G_{\gamma,t-s}(S)^{\frac{n\alpha}{\gamma}}}\Big]+2^{d(n+\epsilon)}A_{n,s}.
\end{align*}
Since $0<\alpha<\gamma<\sqrt{2d}$,
we have $\zeta_n+d<0$.
Let $s$ be large enough such that
$$
  2^{d(n+\epsilon-1)} A_n e^{(\zeta_n+d)s}<1.
$$
Then
$$\sup_{m\in s\N}\E\Big[\frac{G_{\alpha,m}(S)^{n+\epsilon}}
{G_{\gamma,m}(S)^{\frac{n\alpha}{\gamma}}}\Big]
<\infty.$$
For any $t\geq 0$ there exists $m\in s\N$ such that $|t-m|\leq s$.
We have for $x,y\in S$ 
$$|\E[X_t(x)X_t(y)]-\E[X_m(x)X_m(y)]|\leq |t-m|+O(1)\leq s+O(1).$$
Using Proposition \ref{kahane_variant} we obtain
$$
  \E\Big[\frac{G_{\alpha,t}(S)^{n+\epsilon}}
  {G_{\gamma,t}(S)^{\frac{n\alpha}{\gamma}}}\Big]
  \leq \tilde C\;\E\Big[\frac{G_{\alpha,m}(S)^{n+\epsilon}}
  {G_{\gamma,m}(S)^{\frac{n\alpha}{\gamma}}}\Big]
  \leq \tilde C\sup_{m\in s\N}\E\Big[\frac{G_{\alpha,m}(S)^{n+\epsilon}}
{G_{\gamma,m}(S)^{\frac{n\alpha}{\gamma}}}\Big]
$$
for some $\tilde C$ independent of $t$.
Therefore (\ref{shifted_moment}) holds for $n$,
which concludes the proof.

\subsection{Proof of Theorem \ref{moment_gen}}

Define $h_{\epsilon}(x)={\epsilon}^{-d}\int_{\R^d}h(y)\phi(\frac{x-y}{\epsilon})\;\dd y$
for some compactly supported smooth function $\phi$ with $\int_{\R^d} \phi(x)\;\dd x=1$.
Since $g$ is continuous,
we can show that $\E[h_{\epsilon}(x)h_{\epsilon}(y)]+\ln (|x-y|\vee \epsilon)$
is uniformly bounded in $x,y\in S$ and $\epsilon\in(0,\epsilon_0)$
for some small $\epsilon_0$.
By Lemma \ref{star_scale} (i), we can find $A>0$ such that
$$
  |\E[X_t(x)X_t(y)]-\E[h_{e^{-t}}(x)h_{e^{-t}}(y)]|\leq A,\quad
  \forall x,y\in S,\  t> -\ln\epsilon_0.
$$
By Proposition \ref{kahane_variant},
there exists some $C>0$ depending on $A$ such that for all $t\geq 0$,
$$
  \E\Big[\frac{M_{\alpha,h_{e^{-t}}}(S)^n}
  {M_{\gamma,h_{e^{-t}}}(S)
  ^{\frac{n\alpha}{\gamma}}}\Big]\leq C
  \E\Big[\frac{G_{\alpha,t}(S)^n}
  {G_{\gamma,t}(S)^{\frac{n\alpha}{\gamma}}}\Big].
$$
Then Proposition \ref{moment_fini} implies that the left hand side is
uniformly bounded in $t$.
Since $M_{\alpha,h_{e^{-t}}}(S)$ and $M_{\gamma,h_{e^{-t}}}(S)$
converge almost surely to $M_{\alpha,h}(S)$ and $M_{\gamma,h}(S)$ respectively \cite{Berestycki17},
the proof is concluded using Fatou's lemma.

\section{Recursive moment bounds for star-scale invariant kernels}\label{sec_4}


In this section, we establish
some moment bounds for balanced ratios
based on the cascade relation introduced in Section \ref{sec_2}.
To avoid discussions of convergence, we always work on a finite-scale approximation $X_t$
rather than the star-scale invariant field $X$ itself.

We enlarge the class of balanced ratios studied.
Assume that $S=[0,1]^d$ and $\alpha,\gamma\in(0,\sqrt{2d})$ are distinct.
For a $C^1$ function $f\colon S\to\R$ and a Borel set $B\subset S$, define a tilted version of
the balanced ratio
associated to $X_t$ by
\begin{equation}\label{tilt}
  L_{f,t}(B):=\frac{(\int_B e^{\alpha f(x)}
  \;M_{\alpha,X_t}(\dd x))^{\frac{\gamma}{\gamma-\alpha}}}
  {(\int_B e^{\gamma f(x)}\;
  M_{\gamma,X_t}(\dd x))^{\frac{\alpha}{\gamma-\alpha}}}.
\end{equation}
Using Hölder's inequality as in (\ref{holder}),
we can show that $L_{f,t}$ is a subadditive
function on Borel sets.

With these notations,
we can rewrite (\ref{subad}) as
$L_{0,\infty}(S)\leq e^{(\frac{\alpha\gamma}{2}-d)s}
\sum_{i=1}^{e^{ds}}L^{(i)}_{X_s\circ\phi_i,\infty}(S)$
where $L^{(i)}$ is defined by (\ref{tilt}) with $X_t$ replaced by $X^{(i)}$
(see Section \ref{sec_2} for the definition of $X^{(i)}$ and $\phi_i$).
A major issue with the heuristic
is that Assumption \ref{assump1} does not hold exactly,
so the equality $L^{(i)}_{X_s\circ\phi_i,\infty}(S)=L^{(i)}_{0,\infty}(S)$ obtained from it is
merely tentative.
Nevertheless, we can bound the positive moments of $L_{f,t}(S)$ in terms of $f$,
as stated in the following lemma.

\begin{lemma}\label{initial_n}
  For each $n\in \N$
  there exists $C_n>0$ such that
  \begin{equation*}
    \sup_{t\geq 0}\E[L_{f,t}(S)^n]\leq
    C_n(1+\sup_{x\in S}|\nabla f(x)|^{\frac{\alpha\gamma}{2} n}),
    \quad \forall f\in C^1(S).
  \end{equation*}
\end{lemma}

\begin{proof}
  Let $e^s$ be the smallest even integer larger than $\sup_{x\in S}|\nabla f(x)|$.
  Divide $S$ into small cubes
  $\{S_1,\ldots,S_{e^{ds}}\}$ of equal side length $e^{-s}$.
  Then $|f(x)-f(y)|\leq \sqrt{d}$
  for $x,y\in S_i$.
  Using the subadditivity of $L_{f,t}$, we have
  $$
    L_{f,t}(S)\leq \sum_{i=1}^{e^{ds}}L_{f,t}(S_i)
    \leq \sum_{i=1}^{e^{ds}}e^{\frac{\alpha\gamma}
    {\gamma-\alpha}\sup_{x,y\in S_i}|f(x)-f(y)|}L_{0,t}(S_i)
    \leq C_1\sum_{i=1}^{e^{ds}}L_{0,t}(S_i)
  $$
  where $C_1=\exp\big(\frac{\alpha\gamma\sqrt{d}}
  {\gamma-\alpha}\big)$.
  From Lemma \ref{scale},
  the scaling exponent of $L_{0,t}(S)^n$ is given by
  $
    \zeta\big(\frac{n\gamma}{\gamma-\alpha},
    \frac{n\alpha}{\gamma-\alpha}\big)
    =\big(\frac{\alpha\gamma}{2}-d\big)n
  $,
  i.e. there exists some constant $C_2$ such that
  $$\E[L_{0,t}(S_i)^n]\leq C_2e^{(\frac{\alpha\gamma}{2}-d)n s}\E[L_{0,t-s}(S)^n],
  \quad t\in[s,\infty).$$
  Then, we have for all $t\in[s,\infty)$
  $$
    \E[L_{f,t}(S)^n]\leq C_1^ne^{dns}\E[L_{0,t}(S_1)^n]
    \leq  C_1^nC_2e^{\frac{\alpha\gamma}{2}sn}\E[L_{0,t-s}(S)^n]\leq C_3e^{\frac{\alpha\gamma}{2}ns}
  $$
  where $C_3=C_1^nC_2\sup_{t\geq 0}\E[L_{0,t}(S)^n]$
  is bounded from Proposition \ref{moment_fini}.
  Since $\E[L_{f,t}(S)^n]$ is monotone in $t$
  by Proposition \ref{kahane},
  we have $\sup_{t\geq 0}\E[L_{f,t}(S)^n]\leq
  C_3e^{\frac{\alpha\gamma}{2}ns}$.
  Therefore, the lemma is concluded by $e^s\leq \sup_{x\in S}|\nabla f(x)|+2$.
\end{proof}

In fact, we can establish a rigorous version of cascade relation (\ref{real_cascade})
for $L_{f,t}(S)$
and obtain stronger moment bounds by induction.
More precisely, we have the following theorem.

\begin{theorem}\label{moments_recur}
  Let $\alpha,\gamma\in(0,\sqrt{2d})$
  be distinct and $S=[0,1]^d$.
  Suppose that there exist $p,q\in(0,1)$ such that $p-q>\frac{\alpha\gamma}{4}$ and $p>\frac{\alpha\gamma}{2d}$,
  then we can find $A>0$ such that
  the following holds for any $n\in \N$
  \begin{equation}\label{moments_recurII}
    \sup_{t\geq 0}\E[L_{f,t}(S)^n]\leq
    A^n (e^{qn\ln n}\sup_{x\in S}|\nabla f(x)|^{\frac{\alpha\gamma}{2} n}
    +e^{pn\ln n}),\quad\forall f\in C^1(S).
  \end{equation}
\end{theorem}

\begin{proof}
  We prove the theorem by induction.
  Fix a large $n_0\in\N$.
  From Lemma \ref{initial_n},
  there exists $A>0$ such that
  (\ref{moments_recurII}) holds for $n=0,1,\ldots, n_0$.
  Suppose that
  (\ref{moments_recurII}) holds
  for integers less than $n$
  and for above $A$
  where $n> n_0$.
  We show that (\ref{moments_recurII}) also holds for $n$ and the same $A$.
  Throughout the proof, we use the shorthand
  $\sup=\sup_{x\in S}$.
  Define for $n\in \N$ and $t\geq 0$
  $$
    J_{n,t}=\sup_{f\in C^1}
    \E[L_{f,t}(S)^n]
    (e^{qn\ln n}\sup|\nabla f|^{\frac{\alpha\gamma}
    {2}n}+e^{pn\ln n})^{-1}.
  $$
  Then $J_{n,t}$ is finite by Lemma \ref{initial_n}
  and we need to show $J_{n,t}\leq A^n$.

  Fix an even integer $e^s$ and let $t\in[s,\infty)$.
  Divide $S$ into $e^{ds}$ small cubes
  of side length $e^{-s}$.
  As in the proof of Proposition \ref{moment_fini},
  we partition these small cubes into $2^d$ groups
  $\mathcal{S}_1,\ldots,\mathcal{S}_{2^d}$
  such that any two cubes in the same group is
  separated by a distance of at least $e^{-s}$.
  Set $N=|\mathcal{S}_l|=2^{-d}e^{ds}$.
  Suppose $\mathcal{S}_l=\{S_1,\ldots,S_N\}$,
  then by subadditivity
  $L_{f,t}(\mathcal{S}_l)
  \leq \sum_{j=1}^N L_{f,t}(S_j)$.
  It follows from multinomial theorem that
  \begin{align}\label{casc1}
    \E[L_{f,t}(\mathcal{S}_l)^n]
    \leq \sum_{n_1+\cdots+n_N=n}
    \binom{n}{n_1\cdots n_N}
    \E[L_{f,t}(S_1)^{n_1}\cdots L_{f,t}(S_N)^{n_N}]
  \end{align}
  where $n_1,\ldots,n_N$ run over nonnegative integers
  that sum to $n$ and
  $$
    \binom{n}{n_1\cdots n_N}
    =\frac{n!}{n_1!\cdots n_N!}
  $$
  is the multinomial coefficient.

  We seek to control each summand $B_{n_1,\ldots,n_N}=\binom{n}{n_1\cdots n_N}
    \E[L_{f,t}(S_1)^{n_1}\cdots L_{f,t}(S_N)^{n_N}]$.
  Let $\phi_i(x)=e^{-s}x+a_i$
  be the affine map such that $\phi_i(S)=S_i$ and
  define $X^{(i)}=(X_t-X_s)\circ \phi_i$.
  Then a calculation similar to (\ref{subad}) gives
  $$
    L_{f,t}(S_i)=
    e^{(\frac{\alpha\gamma}{2}-d)s}
    L^{(i)}_{(f+X_s)\circ\phi_i}(S)
  $$
  where $L^{(i)}_f(S)$ is defined by (\ref{tilt}) with $X_t$ replaced by $X^{(i)}$,
  i.e.
  \begin{equation*}
    L_f^{(i)}(S)=\frac{(\int_S e^{\alpha f(x)}
    \;M_{\alpha,X^{(i)}}(\dd x))^{\frac{\gamma}{\gamma-\alpha}}}
    {(\int_S e^{\gamma f(x)}\;
    M_{\gamma,X^{(i)}}(\dd x))^{\frac{\alpha}{\gamma-\alpha}}}.
  \end{equation*}
  From Lemma \ref{star_scale},
  $L^{(i)}_f(S)$ are mutually independent
  for $i\in\{1,\ldots,N\}$
  and have the same law as $L_{f,t-s}(S)$.
  In addition, $X^{(i)}$ is independent of $X_s$.
  Therefore
  \begin{align}\label{casc2}
    \E[L_{f,t}(S_1)^{n_1}\cdots L_{f,t}(S_N)^{n_N}|X_s]
    =e^{(\frac{\alpha\gamma}{2}-d)sn}
    \prod_{i=1}^N\E[L^{(i)}_{(f+X_s)
    \circ\phi_i}(S)^{n_i}|X_s].
  \end{align}
  It follows from the definition of $J_{n,t}$ that
  \begin{align}\label{casc3}
    \E[L^{(i)}_{(f+X_s)
    \circ\phi_i}(S)^{n_i}|X_s]
    \leq (e^{qn_i\ln n_i}\mathbb{M}(f,s)^{\frac{\alpha\gamma}{2}n_i}
    +e^{pn_i\ln n_i})J_{n_i,t-s},
  \end{align}
  where $\mathbb{M}(f,s)$ is a random variable defined by
  $$
    \mathbb{M}(f,s)=e^{-s}\sup_{x\in S}|\nabla
    (f+X_s)(x)|.
  $$
  Combining (\ref{casc2}) and (\ref{casc3})
  and using $e^{m\ln m-m}\leq m!
  \leq e^{m\ln m}$ for $m\in \N$,
  we obtain
  \begin{align}\label{summand}
    B_{n_1,\ldots,n_N} \leq e^{n\ln n+n}
    e^{(\frac{\alpha\gamma}{2}-d)sn}
    \E\Big[\prod_{i=1}^N\big(e^{(q-1)n_i\ln n_i}\mathbb{M}(f,s)^{\frac{\alpha\gamma}{2}n_i}
    +e^{(p-1)n_i\ln n_i}\big)\Big]
    \prod_{i=1}^NJ_{n_i,t-s}.
  \end{align}

  In the following, we use  $C$ to denote
  universal constants depending only
  on $\alpha,\gamma,p,q,d$,
  whose values may change from line to line.
  Expanding the product inside the expectation
  in (\ref{summand}),
  we obtain $2^N$ terms and each term is of the form
  \begin{equation}\label{expand_prod}
    \mathbb{M}(f,s)^{\frac{\alpha\gamma}{2}m}
    \prod_{i\in I}e^{(q-1)n_i\ln n_i}
    \prod_{i\in I'}e^{(p-1)n_i\ln n_i}
  \end{equation}
  where $(I,I')$ is a partition
  of $\{1,2,\ldots,N\}$
  and $\sum_{i\in I}n_i=m$.
  We now bound (\ref{expand_prod}) over all possible
  choices of $(n_i)_{i\leq N}$ and $(I,I')$.
  Denote $m'=n-m$.
  Using the convexity of $x\ln x$
  and the assumption that $p\vee q<1$,
  (\ref{expand_prod}) is smaller or equal to
  $$
    \mathbb{M}(f,s)^{\frac{\alpha\gamma}{2}m}
    \exp\left\{(q-1)m\ln \frac{m}{|I|}+
    (p-1)m'\ln\frac{m'}{|I'|}\right\}
  $$
  and the equality holds iff
  $n_i$ is constant for $i$ varying
  in $I$ and in $I'$.
  Note that $|I|+|I'|=N$.
  The above expression
  attains its maximum when
  $$
    |I|=\frac{(1-q)m N}{(1-q)m+(1-p)m'},\quad
    |I'|=\frac{(1-p)m' N}{(1-q)m+(1-p)m'}.
  $$
  Note that $C^{-1}n\leq (1-q)m+(1-p)m'\leq Cn$.
  Therefore, (\ref{expand_prod}) is bounded by
  $$
    C^n \mathbb{M}(f,s)^{\frac{\alpha\gamma}{2}m}
    \exp\left\{((q-1)m+(p-1)m')\ln \frac{n}{N}\right\}.
  $$
  Finally, as a function in $m$,
  the above expression attains its maximum at $m=0$ or $m=n$.
  (\ref{expand_prod}) is bounded by
  \begin{align*}
    C^n(e^{(p-1)n\ln \frac{n}{N}})
    \vee (e^{(q-1)n\ln \frac{n}{N}}
    \mathbb{M}(f,s)^{\frac{\alpha\gamma}{2}n})
    &\leq C^n(e^{(p-1)n\ln \frac{n}{N}}
    +e^{(q-1)n\ln \frac{n}{N}}
    \mathbb{M}(f,s)^{\frac{\alpha\gamma}{2}n})\\
    &=C^ne^{dsn-n\ln n}
    (e^{pn\ln \frac{n}{N}}
    +e^{qn\ln \frac{n}{N}}
    \mathbb{M}(f,s)^{\frac{\alpha\gamma}{2}n}).
  \end{align*}
  There are $2^N$ terms in the expansion.
  Let $e^s=2\lceil \ln n\rceil$
  where $\lceil x\rceil$ is the least integer greater than or equal to $x$.
  Then $N=2^{-d}e^{ds}\leq Cn$
  for all $n\geq 1$.
  Hence,
  \begin{align*}
    \prod_{i=1}^N\big(e^{(q-1)n_i\ln n_i}
    \mathbb{M}(f,s)^{\frac{\alpha\gamma}{2}n_i}
    +e^{(p-1)n_i\ln n_i}\big)
    &\leq 2^NC^ne^{dsn-n\ln n}
    (e^{pn\ln \frac{n}{N}}
    +e^{qn\ln \frac{n}{N}}
    \mathbb{M}(f,s)^{\frac{\alpha\gamma}{2}n})\\
    &\leq C^n e^{dsn-n\ln n}
    (e^{pn\ln \frac{n}{N}}
    +e^{qn\ln \frac{n}{N}}
    \mathbb{M}(f,s)^{\frac{\alpha\gamma}{2}n}).
  \end{align*}
  Then (\ref{summand}) is simplified to be
  \begin{align*}
    B_{n_1,\ldots,n_N}\leq C^n
    e^{\frac{\alpha\gamma}{2}sn}
    \big(e^{pn\ln \frac{n}{N}}
    +e^{qn\ln \frac{n}{N}}
    \E\big[\mathbb{M}(f,s)^{\frac{\alpha\gamma}{2}n}\big]\big)
    \prod_{i=1}^NJ_{n_i,t-s}.
  \end{align*}
  By subadditivity,
  $L_{f,t}(S)^n\leq (\sum_{l=1}^{2^d}L_{f,t}(\mathcal{S}_l))^n
  \leq2^{nd}
  \sum_{l=1}^{2^d}L_{f,t}(\mathcal{S}_l)^n$.
  Combining (\ref{casc1}) and the above
  inequality, we obtain
  \begin{align}\label{casc4}
    \E[L_{f,t}(S)^n]\leq C^n
    e^{\frac{\alpha\gamma}{2}sn}
    \big(e^{pn\ln \frac{n}{N}}
    +e^{qn\ln \frac{n}{N}}
    \E\big[\mathbb{M}(f,s)^{\frac{\alpha\gamma}{2}n}\big]\big)
    \sum_{n_1+\cdots+n_N=n}
    \prod_{i=1}^NJ_{n_i,t-s}.
  \end{align}
  It follows from Proposition \ref{gau_estimate} and
  $s=o(\ln n)$ that
  \begin{align*}
    \E\big[\mathbb{M}(f,s)^{\frac{\alpha\gamma}{2}n}\big]
    &\leq C^n\big(\E\big[(e^{-s}\sup|\nabla X_s|)
    ^{\frac{\alpha\gamma}{2}n}\big]+
    (e^{-s}\sup|\nabla f|\big)
    ^{\frac{\alpha\gamma}{2}n})\\
    &\leq C^n\big(e^{\frac{\alpha\gamma}{4}n\ln n}+
    (e^{-s}\sup|\nabla f|)
    ^{\frac{\alpha\gamma}{2}n}\big).
  \end{align*}
  Hence (\ref{casc4}) becomes
  \begin{align*}
    \E[L_{f,t}(S)^n]\leq C^n
    e^{\frac{\alpha\gamma}{2}sn}
    \big(e^{pn\ln n-pn ds}
    +e^{(q+\frac{\alpha\gamma}{4})n\ln n-qn ds}
    &+e^{qn\ln n-(qd+\frac{\alpha\gamma}{2})sn}
    \sup|\nabla f|
    ^{\frac{\alpha\gamma}{2}n}\big)\\
    &\times\sum_{n_1+\cdots+n_N=n}
    \prod_{i=1}^NJ_{n_i,t-s}.
  \end{align*}
  Dividing both sides by $e^{qn\ln n}\sup|\nabla f|^{\frac{\alpha\gamma}
  {2}n}+e^{pn\ln n}$ and optimizing in $f\in C^1(S)$,
  we obtain the following cascade relation
  for $L_{f,t}$:
  \begin{equation}\label{real_cascade}
    J_{n,t}\leq C^ne^{\frac{\alpha\gamma}{2}sn}
    (e^{-pd sn}+e^{(q-p+\frac{\alpha\gamma}{4})n\ln n-qd sn}
    +e^{-(qd+\frac{\alpha\gamma}{2})sn})
    \sum_{n_1+\cdots+n_N=n}
    \prod_{i=1}^NJ_{n_i,t-s}.
  \end{equation}
  Since $p>\frac{\alpha\gamma}{2d}, p-q>\frac{\alpha\gamma}{4}, q>0$
  and $s=o(\ln n)$, there exists $\epsilon>0$ such that
  $$
    J_{n,t}\leq C^ne^{-\epsilon sn}
    \sum_{n_1+\cdots+n_N=n}
    \prod_{i=1}^NJ_{n_i,t-s}.
  $$
  From induction hypothesis we have $J_{m,t-s}\leq A^{m}$
  for $m<n$.
  There are $\binom{n+N-1}{N-1}$ terms
  in the summation above.
  Since $\binom{n+N-1}{N-1}\leq C^n$, we have
  $$
    J_{n,t}\leq C^ne^{-\epsilon sn}
    (N J_{n,t-s}+C^nA^n)
    \leq C^ne^{-\epsilon sn}(J_{n,t-s}+A^n).
  $$
  Since $s\to\infty$ as $n\to\infty$,
  we can find $n_0$ such that
  $C^ne^{-\epsilon sn}<1/2$ for all $n\geq n_0$.
  In this way we obtain a recursive relation
  $$
    J_{n,t}\leq \frac{1}{2}(J_{n,t-s}+A^n),\quad
    \forall t\in[s,\infty).
  $$
  Since $J_{n,0}\leq 1$ by (\ref{real_holder})
  and $J_{n,t}$ is non-decreasing in $t$
  by Proposition \ref{kahane},
  we have $J_{n,t}\leq A^n$ for all $t$.
  Hence, (\ref{moments_recurII}) holds for $n$ and
  the theorem follows from induction.
\end{proof}

\begin{corollary}\label{laplace}
  Let $\alpha,\gamma\in(0,\sqrt{2d})$
  be distinct and $\alpha\gamma<4$.
  For any $p\in(\frac{\alpha\gamma}{4}\vee \frac{\alpha\gamma}{2d},1)$
  and any $\delta>1$,
  there is some $C>0$ such that
  $$
    \sup_{t\geq 0}\E[e^{\mu L_{f,t}(S)}]\leq C\exp(C\mu^{\delta}
    \sup_{x\in S}|\nabla f(x)|^{\frac{\alpha\gamma}{2}\delta}
    +C\mu^{\frac{1}{1-p}}),\quad\forall \mu\geq 0.
  $$
\end{corollary}

\begin{proof}
  It is a direct consequence of Theorem \ref{moments_recur}
  and the following elementary inequality
  $$
    \sum_{n=0}^{\infty}\frac{\mu^n}{n!}e^{pn\ln n}\leq C_pe^{C_p\mu^{\frac{1}{1-p}}}
    ,\quad \forall \mu\geq 0
  $$
  where $p\in (0,1)$ and $C_p$ is a constant.
  Suppose that the left hand side is smaller than $2$ for $\mu\in[0,\mu_c)$.
  By a change of variable $t\to \lambda t$, we have for $\lambda>0$
  $$
    \int_0^{\infty}t^{\lambda}e^{-t}\;\dd t
    =\lambda^{\lambda+1}\int_0^{\infty} (te^{-t})^{\lambda}\;\dd t
    \geq \lambda^{\lambda+1}\int_1^{\infty} e^{-\lambda t}\;\dd t
    =e^{\lambda\ln\lambda-\lambda}.
  $$
  Hence, for $n\geq 0$ and $\eta=e^{p-p\ln p}$ we have
  $
    e^{pn\ln n}\leq \eta^n \int_0^{\infty}t^{pn}e^{-t}\;\dd t.
  $
  Then
  $$
    \sum_{n=0}^{\infty}\frac{\mu^n}{n!}e^{pn\ln n}\leq \int_0^{\infty}
    \sum_{n=0}^{\infty}\frac{\mu^n}{n!} \eta^n t^{pn}e^{-t}\;\dd t
    =\int_0^{\infty}e^{\eta\mu t^p-t}\;\dd t.
  $$
  By a change of variable $t\to \mu^{\frac{1}{1-p}}t$, we have
  $$
    \int_0^{\infty}e^{\eta\mu t^p-t}\;\dd t
    =\mu^{\frac{1}{1-p}}\int_0^{\infty}
    e^{\mu^{\frac{1}{1-p}}(\eta t^p-t)}\;\dd t
    \leq e^{C_p \mu^{\frac{1}{1-p}}}
  $$
  for some $C_p>2$ if $\mu\geq \mu_c$. This concludes the proof.
\end{proof}

\section{Tail estimates for balanced ratios}\label{sec_5}


In this section, we prove the right tail estimates in the introduction.
Recall that the results are stated in two dimension and $S=[0,1]^2$.
As a preparation, we show the following lemma which translates bounds of the Laplace transform
to bounds of the right tail.
In the following denote $f\lesssim g$ if there is some $C>0$
such that $f(r)\leq C g(r)$ for $r$ sufficiently large.

\begin{lemma}\label{laplace_tail}
  Let $Q$ be a random variable.
  Then for any $0<q<p<1$, we have
  $$
    \ln \E[e^{\mu Q}]\lesssim \mu^{\frac{1}{1-p}}
    \quad\Longrightarrow\quad
    \ln \P(Q>x)\lesssim -x^{\frac{1}{p}}
  $$
  and
  $$
    \mu^{\frac{1}{1-q}}\lesssim \ln \E[e^{\mu Q}]\lesssim
    \mu^{\frac{1}{1-p}}
    \quad\Longrightarrow\quad
    \ln \P(Q>x) \gtrsim -x^\frac{1-q}{q(1-p)}.
  $$
\end{lemma}

\begin{proof}
  By Chebyshev's inequality, there exists $C>0$ such that for $\mu$ sufficiently large,
  $$
    \ln \P(Q>x)\leq \ln \E[e^{\mu Q-\mu x}]\leq
    C\mu^{\frac{1}{1-p}}-\mu x.
  $$
  Let $\mu=(2C)^{\frac{p-1}{p}} x^{\frac{1-p}{p}}$.
  Then we have $\ln \P(Q>x)\leq -C(2C)^{-\frac{1}{p}}x^{\frac{1}{p}}$ for $x$
  sufficiently large.
  
  For the second implication,
  apply Paley-Zygmund inequality to $e^{\mu Q}$ and $\theta\in (0,1)$:
  $$
    \P(e^{\mu Q}>\theta\E[e^{\mu Q}])\geq\frac{
    (1-\theta)^2\E[e^{\mu Q}]^2
    }{\E[e^{2\mu Q}]}.
  $$
  Since $\ln \E[e^{\mu Q}]\gtrsim \mu^{\frac{1}{1-q}}$,
  there exists $c>0$ such that for $\mu$ sufficiently large
  $$
    \P(e^{\mu Q}>\theta\E[e^{\mu Q}])
    \leq \P(Q>c\mu^{\frac{q}{1-q}}).
  $$
  Hence
  $$
    \ln \P(Q>c\mu^{\frac{q}{1-q}})\geq
    2\ln(1-\theta)+2\ln\E[e^{\mu Q}]-\ln\E[e^{2\mu Q}]
    \gtrsim -\mu^{\frac{1}{1-p}},
  $$
  which implies the lower bound for $\ln\P(Q>x)$.
\end{proof}

\subsection{Proof of Theorem \ref{upper}}

By Lemma \ref{star_scale} and the hypothesis on $K$, $K^0-K$ is smooth
on a neighborhood of $S\times S$.
Then by Proposition \ref{perturbation} in the appendix,
there is a $C^1$ Gaussian field $Z$ defined on $S$,
such that
$$
  \E[Z(x)Z(y)]+K^0(x,y)-K(x,y)
$$
is positive definite.
In other words we have an equality in law $Z+X=h+Y$
where $(Z,X)$ and $(h,Y)$ are independent pairs.
In Proposition \ref{kahane}, set $\lambda(\dd x)=M_{\alpha,h}(\dd x), \nu(\dd x)=M_{\gamma,h}(\dd x)$ and compare the field $Y$ with the trivial field $0$.
Then for $\mu\geq 0$ we have
$$
  \E[e^{\mu Q_{h}(S)}]\leq \E[e^{\mu Q_{h+Y}(S)}]=\E[e^{\mu Q_{X+Z}(S)}].
$$
Using the notation in section \ref{sec_4}
and the fact that $\E[Z(x)^2]$ is bounded on $S$,
we have
$$
  Q_{X_t+Z}(S)=
  \frac{\big(\int_S e^{\alpha Z(x)-\frac{\alpha^2}{2}\E[Z(x)^2]}
  \;M_{\alpha,X_t}(\dd x)\big)^{\frac{\gamma}{\gamma-\alpha}}}
  {\big(\int_S e^{\gamma Z(x)-\frac{\gamma^2}{2}\E[Z(x)^2]}\;
  M_{\gamma,X_t}(\dd x)\big)^{\frac{\alpha}{\gamma-\alpha}}}
  \leq A L_{Z,t}(S)
$$
for some constant $A$.
In two dimensions the condition $\alpha\gamma<4$ in Corollary \ref{laplace}
is automatically satisfied.
Then for
any $\frac{\alpha\gamma}{4}<p<1$ and $1<\delta<\frac{4}{\alpha\gamma}$ we have
$$
  \E[e^{\mu Q_{X_t+Z}(S)}|Z]\leq C\exp(C\mu^{\delta}
  \sup_{x\in S}|\nabla Z(x)|^{\frac{\alpha\gamma}{2}\delta}
  +C\mu^{\frac{1}{1-p}}),\quad\forall t,\mu\geq 0.
$$
Denote $\mathbb{M}=\sup_{x\in S}|\nabla Z(x)|$.
From Borell-TIS inequality (see Proposition \ref{Borel} in the appendix),
we can find $C,c>0$ such that
$\P(\mathbb{M}>u)\leq Ce^{-cu^2}$ for all $u\geq 0$.
For $q=\frac{\alpha\gamma}{2}\delta<2$,
\begin{align*}
  \E[\exp(\mu\mathbb{M}^q)]-1&=q\mu\int_{0}^{\infty}
  u^{q-1}e^{\mu u^q}\P(\mathbb{M}>u)\;\dd u\\
  &\leq Cq\mu\int_{0}^{\infty}
  u^{q-1}\exp(\mu u^q-cu^2) \;\dd u\\
  &=Cq\mu^{\frac{2}{2-q}}\int_{0}^{\infty}\lambda^{q-1}
  \exp(\mu^{\frac{2}{2-q}}(\lambda^q-c\lambda^2))\;\dd\lambda.
\end{align*}
On the last line we applied a change of variable $u=\mu^{\frac{1}{2-q}}\lambda$.
We see from Laplace's method that
$\ln\E[\exp(\mu\mathbb{M}^q)]\lesssim \mu^{\frac{2}{2-q}}$.
Hence, there is $C_1>0$ such that
$$
  \E[e^{\mu Q_{X_t+Z}(S)}]\leq C_1\exp(C_1\mu^{\frac{4\delta}
  {4-\alpha\gamma\delta}}+C_1\mu^{\frac{1}{1-p}}),\quad\forall t,\mu\geq 0.
$$
In particular, $e^{\mu Q_{X_t+Z}(S)}$ is uniformly integrable in $t$.
Since $Q_{X_t+Z}(S)$ converges almost surely to $Q_{X+Z}(S)$,
we have
$$
  \E[e^{\mu Q_{X+Z}(S)}]=\lim_{t\to\infty}
  \E[e^{\mu Q_{X_t+Z}(S)}]\leq C_1\exp(C_1\mu^{\frac{4\delta}
  {4-\alpha\gamma\delta}}+C_1\mu^{\frac{1}{1-p}}).
$$
Above all, $\ln\E[e^{\mu Q_h(S)}]\lesssim \mu^{\frac{4\delta}
{4-\alpha\gamma\delta}}+\mu^{\frac{1}{1-p}}$
for any $\frac{\alpha\gamma}{4}<p<1$ and $1<\delta<\frac{4}{\alpha\gamma}$.
Then Lemma \ref{laplace_tail} implies $\ln \P(Q_h(S)>x)\lesssim x^{-\frac{1}{p}}$
for all $p>\frac{\alpha\gamma}{4}$,
which concludes the proof.

\subsection{Proof of Theorem \ref{lower}}

Recall that $h$ admits an independent decomposition $h=X+Z$.
As in the previous proof,
$\E[e^{\mu Q_X(S)}]\leq \E[e^{\mu Q_{h}(S)}]$ holds
for all $\mu\geq 0$.
Moreover, Proposition \ref{kahane} implies that
$\E[e^{\mu Q_{X_t}(S)}]$ is non-decreasing in $t$.
By Corollary \ref{laplace}, $e^{\mu Q_{X_t}(S)}$ is uniformly integrable in $t$.
Hence
$$
  \E[e^{\mu Q_{X_t}(S)}]\leq \lim_{r\to\infty}\E[e^{\mu Q_{X_{r}}(S)}]
  =\E[e^{\mu Q_X(S)}].
$$
On the event $\{\sup_S|X_t|\leq 1\}$,
we have $Q_{X_t}(S)\geq \delta e^{\frac{\alpha\gamma}{2}t}$
where $\delta=e^{-\frac{2\alpha\gamma}{\gamma-\alpha}}$.
Consequently
$$
  \E[e^{\mu Q_{X_t}(S)}]
  \geq \E[e^{\mu Q_{X_t}(S)}\mathbf{1}_{\{\sup_S|X_t|\leq 1\}}]
  \geq \exp(\delta\mu e^{\frac{\alpha\gamma}{2}t})\P(\sup_S|X_t|\leq 1).
$$
From Proposition \ref{small_ball}, there exists $C,c>0$ such that
$\P(\sup_S|X_t|\leq 1)\geq c\exp(-C e^{2t})$.
Above all, it holds for all $t,\mu\geq 0$ that
$$
  \E[e^{\mu Q_h(S)}]\geq c\exp(\delta\mu e^{\frac{\alpha\gamma}{2}t}
  -Ce^{2t}).
$$
Choose $t$ such that $\mu= 2\delta^{-1}Ce^{(2-\frac{\alpha\gamma}{2})t}$.
Then
$
  \ln \E[e^{\mu Q_{h}(S)}]\gtrsim \mu^{\frac{4}{4-\alpha\gamma}}
$
and the lower bound of $\P(Q_h(S)>x)$ is deduced by Lemma \ref{laplace_tail}.

\section{Appendix}

\subsection{Kahane type inequalities}

The monotonicity of GMC
in the sense of Kahane \cite{Kahane85} does not hold for
balanced ratios.
Fortunately, weaker comparison principles remain available,
allowing us to control finite order moments
of the products of GMC.

We first recall some basic GMC calculations.
A function $F\colon(0,\infty)^n\to\R$ is said to have polynomial growth, if there exists
$C,N>0$ such that for all $x_1,\ldots,x_n>0$
$$
  |F(x_1,\ldots,x_n)|\leq C\sum_{i=1}^n
  (|x_i|^N+|x_i|^{-N}).
$$

\begin{lemma}
  Let $(\mu_i)_{i\leq n}$ be Radon measures
  on a compact set $K\subset \R^d$.
  Let $\mathbf{A}(x)$ and $\mathbf{B}(x)$
  be two independent continuous Gaussian fields
  on $K$ with values in $\R^n$.
  For $t\in[0,1]$ and $1\leq i\leq n$,
  set $C_i(t,x)=\sqrt{1-t}A_i(x)+\sqrt{t}B_i(x)$.
  Define
  $$
    W_i(t,x)=e^{C_i(t,x)-\frac{1}{2}\E[C_i(t,x)^2]}\;,\quad
    M_i(t)=\int_KW_i(t,x) \;\mu_i(\dd x).
  $$
  Let $F\colon(0,\infty)^n\to\R$ be a smooth function
  whose derivatives up to order two have polynomial growth.
  Then $\frac{\dd}{\dd t}\E[F(M_1(t),\ldots,M_n(t))]$
  is given by
  \begin{align*}
    \frac{1}{2}\sum_{i,j=1}^n
    \E\Big[\partial_i\partial_jF(M_1(t),\ldots,M_n(t))
    \iint_{K\times K} g_{ij}(x,y)
    W_i(t,x)W_j(t,y)\;
    \mu_i(\dd x)\mu_j(\dd y)\Big]
  \end{align*}
  where $g_{ij}(x,y)=\E[B_i(x)B_j(y)]-\E[A_i(x)A_j(y)]$.
\end{lemma}

\begin{proof}
  All the convergence issues in the computation below are standard results
  of dominated convergence and Borell-TIS inequality \ref{Borel}.
  We leave the readers to check the details.
  Using Leibniz rule we have
  $$
    \E[F(M_1(t),\ldots,M_n(t))]'
    =\sum_{i=1}^n\E[M_i(t)'\partial_i
    F(M_1(t),\ldots,M_n(t))].
  $$
  For $s\in[0,1]$, we have
  by Girsanov transform
  \begin{equation}\label{gir}
    \E[M_i(s)\partial_iF(M_1(t),\ldots,M_n(t))]
    =\int_K \E[\partial_iF(M_1^{i,s,x}(t),\ldots,M_n^{i,s,x}(t))]\;\mu_i(\dd x)
  \end{equation}
  where
  $$
    M_j^{i,s,x}(t)=\int_K e^{\E[C_i(s,x)C_j(t,y)]}
    W_j(t,y)\;\mu_j(\dd y).
  $$
  Notice that $\partial_s|_{s=t}\E[C_i(s,x)C_j(t,y)]=g_{ij}(x,y)/2$. Then
  $$
    \partial_s|_{s=t} M_j^{i,s,x}(t)=
    \frac{1}{2}\int_K g_{ij}(x,y)
    e^{\E[C_i(t,x)C_j(t,y)]}
    W_j(t,y)\;\mu_j(\dd y).
  $$
  Differentiate both sides of (\ref{gir})
  with respect to $s$ at $s=t$,
  and sum the result over $i$.
  Then $\E[F(M_1(t),\ldots,M_n(t))]'$ becomes
  $$
    \frac{1}{2}\sum_{i,j=1}^n\iint g_{ij}(x,y)
    \E\big[\partial_i\partial_j F
    (M_1^{i,t,x}(t),\ldots,M_n^{i,t,x}(t))
    e^{\E[C_i(t,x)C_j(t,y)]}
    W_j(t,y)\big]\;\mu_i(\dd x)\mu_j(\dd y).
  $$
  Applying Girsanov transform 
  \begin{align*}
    \E\big[\partial_i\partial_j &F
    (M_1^{i,t,x}(t),\ldots,M_n^{i,t,x}(t))
    e^{\E[C_i(t,x)C_j(t,y)]}
    W_j(t,y)\big]\\
    &=\E\big[\partial_i\partial_j F
    (M_1^{i,t,x}(t),\ldots,M_n^{i,t,x}(t))
    W_i(t,x)W_j(t,y)\big],
  \end{align*}
  we obtain the desired result.
\end{proof}

As a special case we have the following
corollary.

\begin{corollary}\label{kahane_cal}
  Let $(p_j)_{j\leq n}$ and $(\gamma_j)_{j\leq n}$ be real parameters,
  and let $(\mu_j)_{j\leq n}$ be Radon measures
  on a compact set $K\subset \R^d$.
  Suppose that $X$ and $Y$ are two independent continuous
  Gaussian fields on $K$.
  Define for $t\in[0,1]$
  $$
    \Phi(t)=\prod_{j=1}^n\Big(\int_K e^{\gamma_j
    Z(t,x)-\frac{\gamma^2_j}{2}\E[Z(t,x)^2]}\;\mu_j(\dd x) \Big)^{p_j}
  $$
  where $Z(t,x)=\sqrt{1-t}X(x)+\sqrt{t}Y(x)$.
  Then
  \begin{align*}
    \frac{\dd}{\dd t}\E[\Phi(t)]=
    \frac{1}{2}\E\Big[&\Phi(t)\iint_{K\times K}
    \big(\E[Y(x)Y(y)]-\E[X(x)X(y)]\big)\\
    &\times\Big(\sum_{j=1}^{n} \gamma_j^2p_j(p_j-1)\;\mu_{i,t}(\dd x)\mu_{i,t}(\dd y)
    +2\sum_{i<j}\gamma_i\gamma_jp_ip_j\;\mu_{i,t}(\dd x)\mu_{j,t}(\dd y)\Big)\Big]
  \end{align*}
  where $\mu_{i,t}(\dd x)\propto e^{\gamma_i Z(t,x)-\frac{\gamma_i^2}{2}
  \E[Z(t,x)^2]}\;\mu_i(\dd x)$ are random probability
  measures on $K$.
\end{corollary}

\begin{proof}
  Set $\mathbf{A}(x)=(\gamma_1X(x),\ldots,\gamma_n X(x))$ and
  $\mathbf{B}(x)=(\gamma_1Y(x),\ldots,\gamma_n Y(x))$.
  The result follows by applying the above lemma
  to $F(x_1,\ldots,x_n)=x_1^{p_1}\cdots x_n^{p_n}$.
\end{proof}

Now we give the statement and the proof of
two Kahane-type inequalities.

\begin{proposition}\label{kahane_variant}
  Let $(p_j)_{j\leq n}$ and $(\gamma_j)_{j\leq n}$ be real parameters.
  For any $A>0$, there exists a constant $C$ such that the following holds:
  Let $X$ and $Y$ be two continuous
  Gaussian fields on a compact set $K\subset\R^d$ satisfying
  $$
    |\E[X(x)X(y)]-\E[Y(x)Y(y)]|\leq A,\quad
    \forall x,y\in K.
  $$
  Then for any Radon measures $(\mu_j)_{j\leq m}$
  $$
    \E\left[\prod_{j=1}^n\Big(\int_K e^{\gamma_j
    X(x)-\frac{\gamma^2_j}{2}\E[X(x)^2]}\;\dd\mu_j \Big)^{p_j}\right]
    \leq C\E\left[\prod_{j=1}^n\Big(\int_K e^{\gamma_j
    Y(x)-\frac{\gamma^2_j}{2}\E[Y(x)^2]}\;\dd\mu_j \Big)^{p_j}\right].
  $$
\end{proposition}

\begin{proof}
  Use the notation in Corollary \ref{kahane_cal}.
  Define
  $$
    C=A\sum_{j=1}^{n} |\gamma_j^2p_j(p_j-1)|
    +2A\sum_{i<j}|\gamma_i\gamma_jp_ip_j|.
  $$
  Then
  $|\E[\Phi(t)]'|\leq C\E[\Phi(t)]$ for all $t\in(0,1)$.
  Consequently $\E[\Phi(0)]\leq e^C\E[\Phi(1)]$
  which gives the desired inequality.
\end{proof}

\begin{proposition}\label{kahane}
  Let $\alpha,\gamma$ be distinct positive numbers.
  Let $\lambda,\nu$ be two Radon measures on a compact set $K\subset\R^d$.
  Define for a continuous Gaussian field $X$ on $K$
  $$
    \mathcal{Q}_X=\frac{(\int_K e^{\alpha X(x)-
    \frac{\alpha^2}{2} \E[X(x)^2]}\;\dd\lambda)^{\frac{\gamma}{\gamma-\alpha}}}
    {(\int_K e^{\gamma X(x)-\frac{\gamma^2}{2}\E[X(x)^2]}\;\dd\nu)^{\frac{\alpha}{\gamma-\alpha}}}.
  $$
  Suppose that $X$ and $Y$ are two continuous Gaussian fields on $K$, such that
  $$
    \E[Y(x)Y(y)]-\E[X(x)X(y)]
  $$
  is positive definite.
  Then for all non-decreasing convex
  function $F\colon(0,\infty)\to \R$,
  we have
  $$
    \E[F(\mathcal{Q}_X)]\leq \E[F(\mathcal{Q}_Y)].
  $$
\end{proposition}

\begin{proof}
  Since $\E[Y(x)Y(y)]-\E[X(x)X(y)]$ is positive definite,
  we can couple $X,Y$ such that $Y-X$ is independent of $X$.
  By Jensen's inequality we have
  $
    \E[F(\mathcal{Q}_Y)|X]\geq F(\E[\mathcal{Q}_Y|X])
  $.
  If we can prove $\E[\mathcal{Q}_Y|X]\geq \mathcal{Q}_X$,
  then the result follows from the monotonicity of $F$.

  By the independence of $Y-X$ and $X$,
  after absorbing $X$ into $\lambda$ and $\nu$,
  it suffices to show $\E[\mathcal{Q}_0]\leq \E[\mathcal{Q}_{Y-X}]$.
  Let $Z_t=\sqrt{t}(Y-X)$.
  From Corollary \ref{kahane_cal} there exists
  random probability measures $\lambda_t,\nu_t$
  such that
  \begin{align*}
    \frac{\dd}{\dd t}\E[\mathcal{Q}_{Z_t}]=
    \frac{\alpha\gamma}{2(\alpha-\gamma)^2}\E\Big[\mathcal{Q}_{Z_t}\iint g(x,y)
    (\alpha\lambda_t(\dd x)-\gamma\nu_t(\dd x))
    (\alpha\lambda_t(\dd y)-\gamma\nu_t(\dd y))\Big].
  \end{align*}
  Here $g(x,y)$ is the covariance kernel of $Y-X$.
  By the positive-definiteness of $g(x,y)$,
  $\E[\mathcal{Q}_{Z_t}]$ is non-decreasing in $t$,
  which implies $\E[\mathcal{Q}_0]\leq \E[\mathcal{Q}_{Y-X}]$.
\end{proof}

\subsection{Perturbation of covariance kernels}

\begin{proposition}\label{perturbation}
  For any $f\in C^{\infty}_c(\R^d\times\R^d)$ compactly supported,
  there exists a $C^1$ Gaussian field $Z$ on $\R^d$,
  such that $\E[Z(x)Z(y)]+f(x,y)$ is a covariance kernel.
\end{proposition}

\begin{proof}

  The following calculations are borrowed from Lemma 4.6 in \cite{JunnilaSaksmanWebb19_decomposition}.
  For $\phi\in L^1(\R^d)$ denote the
  Fourier transform $\widehat\phi(\xi)=\int_{\R^d}
  e^{-2\pi ix\xi}\phi(x)\;\dd x$ and the
  Sobolev norm for $s\in\R$
  $$
    \Vert \phi\Vert_{H^s}^2=
    \int_{\R^d}(1+|\xi|^2)^s|\widehat\phi(\xi)|^2\;\dd\xi.
  $$
  Let $R(x)$ be a positive definite function on $\R^d$ such that
  $\widehat R(\xi)=(1+|\xi|^2)^{-\frac{d}{2}-2}$.
  By Plancherel formula,
  it holds for $\phi\in C_c^{\infty}(\R^d)$ that
  $$
    \iint_{\R^{2d}} R(x-y)\phi(x)\phi(y)\;\dd x\dd y
    =\int_{\R^d} \widehat R(\xi)|\widehat\phi(\xi)|^2\;\dd\xi
    =\Vert \phi\Vert_{H^{-d/2-2}}^2.
  $$
  Since $\widehat f$ is rapidly decreasing, Cauchy-Schwartz inequality implies that
  \begin{align*}
    \left|\iint_{\R^{2d}} f(x,y)\phi(x)\phi(y)\;\dd x\dd y\right|
    =\left|\iint_{\R^{2d}} \widehat f(\xi,\eta)\overline{\widehat\phi}(\xi)
    \overline{\widehat\phi}(\eta)\;\dd\xi\dd\eta\right|
    \leq A \Vert\phi\Vert_{H^{-d/2-2}}^2
  \end{align*}
  for some constant $A$ determined by $f$.
  Since $R(x)$ is $C^3$ by definition, there exists a $C^1$ Gaussian field $Z$
  such that $\E[Z(x)Z(y)]=(A+1) R(x-y)$ (see Section 1.4.2 in \cite{Adler2007RandomFA}
  for a discussion of regularity properties of Gaussian fields).
  Since $\E[Z(x)Z(y)]+f(x,y)$ is nonnegative when
  integrated against $\phi(x)\phi(y)\dd x\dd y$,
  it is by definition positive definite.
\end{proof}

\subsection{Maximum of Gaussian process}

We record here two inequalities about the maximum
of a Gaussian process.
The readers can consult \cite{Adler2007RandomFA} for a detailed treatment.
In the following,
$(f(t))_{t\in T}$ is a separable centered Gaussian process
on a topological space $T$ (for instance, a continuous Gaussian field on a
compact subset of $\R^d$).

\begin{proposition}[Borell-TIS inequality]\label{Borel}
  Suppose that $(f(t))_{t\in T}$ is almost surely bounded.
  Then $\E[\sup_{t\in T}f(t)]$ is finite
  and for all $u\geq 0$
  and $\sigma_T^2=\sup_{t\in T}\E[f(t)^2]$,
  $$
    \P(\sup_{t\in T} f(t)-\E[\sup_{t\in T} f(t)]>u)
    \leq \exp(-u^2/2\sigma_T^2).
  $$
\end{proposition}

\begin{proposition}[Dudley's theorem]\label{Dudley}
  
  Let $d$ be the pseudometric on $T$ defined by
  $$
    d(t,s)=\sqrt{\E[(f(t)-f(s))^2]}.
  $$
  Suppose that $T$ is $d$-compact and let $N(\epsilon)$ be the smallest number
  of $d$-balls of radius $\epsilon$ needed to cover $T$.
  Then there exists a universal constant $K$ such that
  $$
    \E\big[\sup_{t\in T}f(t)\big]\leq K\int_0^{\infty}
    \sqrt{\ln N(\epsilon)}\;\dd\epsilon.
  $$
\end{proposition}

\subsection{Estimates for the star-scale invariant fields}

We show two propositions that are used in the proof of tail estimates.

\begin{proposition}\label{gau_estimate}
  Let $(X_t(x))_{t\geq 0,x\in\R^d}$ be the
  white noise decomposition in Section \ref{sec_2}
  and let $S=[0,1]^d$.
  For each $\epsilon>0$,
  there exists some $C>0$ such that
  $$
    \E[(e^{-s}\sup_{x\in S} |\nabla X_s(x)|)^m]\leq C^m(
    e^{sm}+e^{\frac{1}{2}m\ln m}),
    \quad \forall s\geq 0,m\geq \epsilon.
  $$
\end{proposition}

\begin{proof}
  Denote $\sup=\sup_{x\in S}$.
  Then
  $$
    \sup|\nabla X_s|\leq \sum_{i=1}^d\sup|\partial_j X_s|
    \leq \sum_{i=1}^d(\sup\partial_j X_s)_++\sum_{i=1}^d(-\sup\partial_j X_s)_+.
  $$
  Note that $(-X_s(x))_{x\in\R^d}\stackrel{\text{law}}{=}(X_s(x))_{x\in\R^d}$ and the law of $X_s$ is rotationally invariant.
  Let $f_s(x)=e^{-s}\partial_1 X_s(x)$.
  It suffices to find some constant $C$ such that
  \begin{equation}\label{gau_1}
    \E[(\sup f_s(x))_+^m]\leq C^m(e^{sm}+e^{\frac{1}{2}m\ln m}),\quad
    \forall s\geq 0,m\geq \epsilon.
  \end{equation}

  Let $h(x)=-\partial_1^2\rho(x)$.
  By definition, the covariance kernel of $f_s$ is given by
  $$
    \E[f_s(x)f_s(y)]=\int_0^s
    h(e^u(x-y))e^{2(u-s)} \;\dd u.
  $$
  Let $d_s$ and $N_s(\epsilon)$ denote respectively
  the Dudley metric and the associated covering number of $S$
  induced by $f_s$ (see Proposition \ref{Dudley}).
  Since $h$ is positive definite, there exists $A>0$ such that
  $h(0)-h(x)\leq A^2|x|^2$ for all $x\in\R^d$.
  For $x,y\in \R^d$ we have
  \begin{align*}
    d_s(x,y)^2=2\int_0^s(h(0)-&h(e^u(x-y)))e^{2(u-s)}\;\dd u\\
    &\leq 2\int_0^s A^2e^{4u-2s}|x-y|^2\;\dd u
    \leq A^2 e^{2s}|x-y|^2.
  \end{align*}
  Hence, each Euclidean ball of radius $A^{-1}e^{-s}\epsilon$ is contained
  in a $d_s$-ball of radius $\epsilon$.
  Consequently
  $
    N_s(\epsilon)\leq \lceil Ae^{s}\epsilon^{-1}\rceil^d
  $
  where $\lceil x\rceil$ is the least integer greater than or equal to $x$.
  Dudley's theorem \ref{Dudley} implies that
  \begin{align*}
    \E[\sup f_s(x)]\leq K\int_0^{\infty}\sqrt{d\ln
    \lceil Ae^{s}\epsilon^{-1}\rceil}\;\dd\epsilon
    =K' e^s
  \end{align*}
  for some constant $K'$.
  Moreover, since
  $
    \sup \E[f_s(x)^2]\leq h(0)/2
  $
  and Borell-TIS inequality \ref{Borel},
  we have for some $\epsilon>0$ 
  $$
    \P(\sup f_s(x)-\E[\sup f_s(x)]\geq u)
    \leq e^{-\epsilon u^2},\quad\forall u\geq 0,s\geq 0.
  $$
  Denote $Y=\sup f_s(x)$.
  We have for $m\geq 2$ and some $C_1>0$
  \begin{align*}
    \E[Y_+^m] &=\int_0^{\infty} mu^{m-1}\P(Y\geq u)\;\dd u\\
    &=\int_0^{\E[Y]} mu^{m-1}\P(Y\geq u)\;\dd u
    +\int_0^{\infty}m(u+\E[Y])^{m-1}\P(Y\geq u+\E[Y])\;\dd u\\
    &\leq\int_0^{\E[Y]} mu^{m-1}\;\dd u
    +\int_0^{\infty}m(u+\E[Y])^{m-1}e^{-\epsilon u^2}\;\dd u\\
    &\leq \E[Y]^m+\int_0^{\infty}m 2^{m-1}
    (u^{m-1}+\E[Y]^{m-1})e^{-\epsilon u^2}\;\dd u\\
    &\leq C_1^m(e^{sm}+\int_0^{\infty}
    u^{m-1}e^{-\epsilon u^2}\;\dd u).
  \end{align*}
  By a change of variable $u=\sqrt{m-1}\;t$,
  there is some $C_2>0$ such that
  \begin{align*}
    \int_0^{\infty}
    u^{m-1}e^{-\epsilon u^2}\;\dd u= (m-1)^{\frac{m}{2}}
    \int_0^{\infty} (te^{-\epsilon t^2})^{m-1}\;\dd t
    \leq C_2^m m^{\frac{m}{2}}.
  \end{align*}
  Above all, the following holds for some $C_3>0$:
  \begin{align*}
    \E[(\sup f_s(x))_+^m]\leq C_3^m(e^{sm}+e^{\frac{1}{2}m\ln m})
    ,\quad \forall s\geq 0,m\geq 2.
  \end{align*}
  For $\epsilon\leq m<2$ we use
  Hölder's inequality to deduce
  $$
    \E[(\sup f_s(x))_+^m]\leq
    \E[(\sup f_s(x))_+^2]^{\frac{m}{2}}
    \leq C_3^m(e^{sm}+e^{\frac{m}{2}\ln2})
    \leq C^m(e^{sm}+e^{\frac{1}{2}m\ln m})
  $$
  where $C=\sqrt{2/\epsilon}\;C_3$.
  Hence, (\ref{gau_1}) holds with $C$.
\end{proof}

The proposition below can be seen as a variant of
the small ball probability estimates.
The argument is adapted from \cite{Ledoux1996}.

\begin{proposition}\label{small_ball}
  Let $(X_t(x))_{t\geq 0,x\in\R^d}$ be the
  white noise decomposition in Section \ref{sec_2}
  and let $S=[0,1]^d$.
  There exists $C,c>0$ such that
  $$
    \P(\sup_{x\in S}|X_t(x)|\leq 1)\geq
    c\exp(-Ce^{dt}),\quad\forall t\geq 0.
  $$
\end{proposition}

\begin{proof}
  Denote the Dudley metric by
  $d_t(x,y)=\sqrt{\E[|X_t(x)-X_t(y)|^2]}$.
  There exists $A>0$ such that
  $\rho(0)-\rho(x)\leq A^2|x|^2$ for all $x\in\R^d$.
  Therefore, for $x,y\in \R^d$ we have
  \begin{align*}
    d_t(x,y)^2=2\int_0^t
    \rho(0)-\rho(e^u (x-y))\;\dd u
    \leq 2\int_0^t A^2e^{2u}|x-y|^2\;\dd u
    \leq A^2 e^{2t}|x-y|^2.
  \end{align*}

  For any $n\geq 1$,
  there exists a finite covering of $S$ by $d_t$-balls of radius $2^{-2n}$.
  Let $T_n$ be the set of the centers.
  Since $d_t(x,y)\leq Ae^t|x-y|$,
  we can choose $T_n$ such that $|T_n|\leq A_1e^{dt} 2^{2dn}$ for some universal constant $A_1>0$.
  Define a map $p_n\colon S\to T_n$ such that
  $d_t(x,p_n(x))\leq 2^{-2n}$
  for all $x\in S$.
  Assume $T_0=\{0\}$ and $p_0(x)=0$.
  Define for $n\geq 1$
  $$
    \mathcal{Y}_n=\{X_t(x)-X_t(p_{n-1}(x)):
    x\in T_n\}.
  $$
  Note that $\E[Y^2]\leq A_2 2^{-4n}$ for all $Y\in\mathcal{Y}_n$
  and some universal constant $A_2$.
  For each $x\in T_n$, $X_t(x)$ can be written as
  $$X_t(x)=X_t(0)+\sum_{k=1}^n Y_k$$
  where $Y_k\in\mathcal{Y}_k$
  for $k=1,\ldots,n$.
  Since $\bigcup_{n\geq 1} T_n$ is dense in $S$
  and $X_t$ is continuous, we have
  $$
    \{|X_t(0)|\leq 2^{-1}\}\bigcap\bigcap_{n\geq 1}
    \{|Y|\leq 2^{-n-1},\ \forall
    Y\in \mathcal{Y}_n\}
    \subset \{\sup_{x\in S}|X_t(x)|\leq1\}.
  $$
  By Gaussian correlation inequality \cite{Latała2017},
  we obtain
  \begin{align*}
    \P(\sup_{x\in S}|X_t(x)|\leq 1)&\geq
    \P(|X_t(0)|\leq 2^{-1})\prod_{n\geq 1}
    \prod_{Y\in\mathcal{Y}_n}\P(|Y|\leq 2^{-n-1})\\
    &\geq \P(|X_t(0)|\leq 2^{-1})
    \prod_{n\geq 1}\P(|N|\leq A_2^{-1/2}2^{n-1})
    ^{A_1e^{dt}2^{2dn}}
  \end{align*}
  where $N$ is a standard Gaussian variable.
  Using the elementary inequality $1-x\geq e^{-2x}$
  when $x\in[0,1/2]$,
  the product $M=\prod_{n\geq 1}\P(|N|\leq A_2^{-1/2}2^{n-1})^{A_12^{2dn}}$
  is nonzero if the series
  $$
    \sum_{n\geq 1}2^{2dn}\P(|N|> A_2^{-1/2}2^{n-1})
  $$
  converges.
  This can be shown via
  the inequality
  $
    P(|N|> u)\leq e^{-u^2/2}
  $ for $u\geq 1$.
  Let $C=-\ln M\in(0,\infty)$.
  Since $\E[X_t(0)^2]=t$, we have
  $$
    \P(\sup_{x\in S}|X_t(x)|\leq 1)
    \geq \P(|N|\leq (2\sqrt{t})^{-1})
    \exp(-Ce^{dt}).
  $$
  This gives the desired lower bound.
\end{proof}

\hspace{10 cm}

\bibliography{ref.bib}

\begin{thebibliography}{10}

\bibitem{Adler2007RandomFA}
Robert~J. Adler and Jonathan~E. Taylor.
\newblock Random fields and geometry.
\newblock 2007.

\bibitem{barashkov2025smalldeviationsgaussianmultiplicative}
Nikolay Barashkov, Joona Oikarinen, and Mo~Dick Wong.
\newblock Small deviations of gaussian multiplicative chaos and the free energy of the two-dimensional massless sinh--gordon model, 2025.

\bibitem{Berestycki17}
Nathana\"{e}l Berestycki.
\newblock An elementary approach to {G}aussian multiplicative chaos.
\newblock {\em Electron. Commun. Probab.}, 22:Paper No. 27, 12, 2017.

\bibitem{bonnefont2023lefttailsubcriticalderivative}
Benjamin Bonnefont and Vincent Vargas.
\newblock The left tail of the subcritical derivative martingale in a branching random walk, 2023.

\bibitem{chen2025lefttailsubcriticalderivative}
Xinxin Chen, Yichao Huang, and Heng Ma.
\newblock Left tail of the subcritical derivative martingale in a branching wiener process, 2025.

\bibitem{DuplantierSheffield11}
Bertrand Duplantier and Scott Sheffield.
\newblock Liouville quantum gravity and {KPZ}.
\newblock {\em Invent. Math.}, 185(2):333--393, 2011.

\bibitem{Fyodorov_2008}
Yan~V Fyodorov and Jean-Philippe Bouchaud.
\newblock Freezing and extreme-value statistics in a random energy model with logarithmically correlated potential.
\newblock {\em Journal of Physics A: Mathematical and Theoretical}, 41(37):372001, August 2008.

\bibitem{10.1214/18-ECP168}
Christophe Garban, Nina Holden, Avelio Sep{\'u}lveda, and Xin Sun.
\newblock {Negative moments for Gaussian multiplicative chaos on fractal sets}.
\newblock {\em Electronic Communications in Probability}, 23(none):1 -- 10, 2018.

\bibitem{guillarmou2024conformalbootstrapsurfacesboundary}
Colin Guillarmou, Rémi Rhodes, and Baojun Wu.
\newblock Conformal bootstrap for surfaces with boundary in liouville cft. part 1: Segal axioms, 2024.

\bibitem{JunnilaSaksmanWebb19_decomposition}
Janne Junnila, Eero Saksman, and Christian Webb.
\newblock Decompositions of log-correlated fields with applications.
\newblock {\em Ann. Appl. Probab.}, 29(6):3786--3820, 2019.

\bibitem{Kahane85}
J-P. Kahane.
\newblock Sur le chaos multiplicatif.
\newblock {\em Ann. Sci. Math. Qu\'{e}bec}, 9(2):105--150, 1985.

\bibitem{KAHANE1976131}
J.-P Kahane and J~Peyrière.
\newblock Sur certaines martingales de benoit mandelbrot.
\newblock {\em Advances in Mathematics}, 22(2):131--145, 1976.

\bibitem{mabuchi}
Hubert Lacoin, R{\'e}mi Rhodes, and Vincent Vargas.
\newblock {Path integral for quantum Mabuchi K-energy}.
\newblock {\em Duke Mathematical Journal}, 171(3):483 -- 545, 2022.

\bibitem{Latała2017}
Rafa{\l} Lata{\l}a and Dariusz Matlak.
\newblock {\em Royen's Proof of the Gaussian Correlation Inequality}, pages 265--275.
\newblock Springer International Publishing, Cham, 2017.

\bibitem{Ledoux1996}
Michel Ledoux.
\newblock {\em Isoperimetry and Gaussian analysis}, pages 165--294.
\newblock Springer Berlin Heidelberg, Berlin, Heidelberg, 1996.

\bibitem{Remy20}
G.~Remy.
\newblock The {F}yodorov-{B}ouchaud formula and {L}iouville conformal field theory.
\newblock {\em Duke Math. J.}, 169(1):177--211, 2020.

\bibitem{Robert2008GaussianMC}
Raoul Robert and Vincent Vargas.
\newblock Gaussian multiplicative chaos revisited.
\newblock {\em Annals of Probability}, 38:605--631, 2008.

\bibitem{SHAMOV20163224}
Alexander Shamov.
\newblock On gaussian multiplicative chaos.
\newblock {\em Journal of Functional Analysis}, 270(9):3224--3261, 2016.

\end{thebibliography}
\bibliographystyle{plain}

\end{document}